\documentclass[a4,10pt]{article} \usepackage{amssymb}
 \usepackage{amsthm}\usepackage{amsmath}\usepackage{latexsym}\usepackage{srcltx}

\newtheorem{teo}{Theorem}[section]
\newtheorem{oss}[teo]{Remark}

\newtheorem{Prop}[teo]{Proposition}
\newtheorem{lemma}[teo]{Lemma}

\newtheorem{Defi}[teo]{Definition}

\newtheorem{corollario}[teo]{Corollary}
\newtheorem{no}[teo]{Notation}

\newcommand{\cc}{_{^{_\HH}}}

\newcommand{\res}{\mathop{\hbox{\vrule height 7pt width .5pt depth 0pt
\vrule height .5pt width 6pt depth 0pt\,}}\nolimits}

\def \op{^\perp}

\newcommand{\LL}{\mathop{\hbox{\vrule height .5pt width 6pt depth
0pt \vrule height 7pt width .5pt depth 0pt\,}}\nolimits}
\newcommand{\rr}{_{^{_\mathcal{R}}}}

\newcommand{\ngr}{_{^{_{\mathrm Gr}}}}

\def \cin{{\mathbf{C}^{\infty}}}

\def\dim {\mathrm{dim}}
\def\dc {d_{\!C\!C}}

\def\ss{_{^{_{\HS}}}}
\def\eu {_{^{_{\rm Eu}}}}

\def\g{h\cc}

\def\dg{\textit{grad}\cc}
\def\qq{\textit{grad}\ss}
\def\ts{_{^{\!{_{\TT\!S}}}}}

\def \per  {\sigma^{2n}\cc}
\def \perh {\sigma^{2n}\cc}

\def \perht {(\sigma^{2n}\cc)_t}

\def \cn{\textit{w}}

\def\WW{\widetilde{W}}

\def\Sph{\mathbb{S}_{\mathbb{H}^n}}
\def\Spu{\mathbb{S}_{\mathbb{H}^1}}
\def\UU{\mathcal{U}}

\def \nn{\nu\cc}

\def \XX{\mathfrak{X}}

\def \MS{\mathcal{H}\cc}

\def \P{{\mathcal{P}}}

\def \PH{\P\cc}

\def \Om{\Omega}

\def \R{\mathbb{R}}
\def \Rn{\mathbb{R}^{\DN}}

\def \div{\mathit{div}}

\def \GG{\mathbb{G}}

\def \perth{\left(\perh\right)_t}

\def\divh{\div_{\!^{_\HH}}}

\def\lh{\mathcal{L}\ss}

\def\lg{\mathcal{D}\ss}

\def\tsc{\nabla^{^{_{\TT\!{S}}}}}
\def\gs{\nabla^{_{\HS}}}
\def\gc{\nabla^{_{\HH}}}

\def \Tor{{\textsc{T}}}

\def \cn{\textit{w}}
\def\WW{\widetilde{W}}

\def\UU{\mathcal{U}}

\def\UU{\mathcal{U}}

\def \nn{\nu_{_{\!\HH}}}

\def \cont{{\mathbf{C}}}
\def \Om{\Omega}
\def \Rn{\mathbb{R}^{n}}
\def \R{\mathbb{R}}

\def \cji {c_{j\,i}(x)}
\def \C { C(x):=[\cji]_{j,i},\,\, {j=1,\ldots,m \,,\, i=1,\ldots,n}}

\def \X {X=(X_{1}, \ldots, X_{m_1})}

\def \X0 {X_{1}(0)\!=\!\partial_{x_{1}}, \ldots, X_{m_1}(0)\!=\!\partial_{x_{m_1}}}

\def \HG {\HH\GG}
\def \HS {\HH\!{S}}

\def \TG {\mathit{T}\GG}
\def \HH {\mathit{H}}

\def \TT {\mathit{T}}
\def \TS {\mathit{T}S}

\def \grad{\textit{grad}}
\def \C0H{\mathbf{C}_{0}^{\infty}(U,\HG)}
\def \C00{\mathbf{C}_{0}^{\infty}(U)}
\def \C01{\mathbf{C}_{0}^{1}(U)}

\def \L1{d\,\mathcal{L}^n}

\def \H1{\mathcal{H}_{{\bf cc}}^{1}}

\def \exp{\textsl{exp\,}}
\def \log{\textsl{log\,}}

\def \Om{\Omega}
\def \Rn{\mathbb{R}^{n}}
\def \R{\mathbb{R}}

\def \cji {c_{j\,i}(x)}
\def \C { C(x):=[\cji]_{j,i},\,\, {j=1,\ldots,m \,,\, i=1,\ldots,n}}

\def \GG{\mathbb{G}}

\def \X {X=(X_{1}, \ldots, X_{m_1})}

\def \X0 {X_{1}(0)\!=\!\partial_{x_{1}}, \ldots, X_{m_1}(0)\!=\!\partial_{x_{m_1}}}

\def \HG {\mathit{H}}

\def \C0H{\mathbf{C}_{0}^{\infty}(\Om,\HG)}
\def \C00{\mathbf{C}_{0}^{\infty}(\Om)}
\def \C01{\mathbf{C}_{0}^{1}(\Om)}

\def \exp{\textsl{exp\,}}

\def\GG{\mathbb{G}}

\begin{document}

\title{Stability of Heisenberg Isoperimetric Profiles}
\author{Francescopaolo Montefalcone
\thanks{F. M. was partially supported by CIRM, Fondazione Bruno Kessler, Trento, and by the Fondazione CaRiPaRo Project ``Nonlinear Partial Differential Equations: models, analysis, and control-theoretic problems".}}

\maketitle

\date{}

\begin{abstract}In the context of sub-Riemannian {Heisenberg
groups} $\mathbb{H}^n,\,n\geq 1$, we shall study Isoperimetric
Profiles, which are closed compact hypersurfaces having constant
horizontal
 mean curvature, very similar to
ellipsoids. Our main goal is to study the stability of Isoperimetric Profiles.\\{\noindent \scriptsize \sc Key words and phrases:}
{\scriptsize{\textsf {Carnot groups; Sub-Riemannian geometry;
hypersurfaces.}}}\\{\scriptsize\sc{\noindent Mathematics Subject
Classification:}}\,{\scriptsize \,49Q15, 46E35, 22E60.}
\end{abstract}

\tableofcontents

\section{Introduction}\label{basicsHYP}

In the last few years sub-Riemannian Carnot groups have become a large research field in Analysis and Geometric Measure Theory; see, for
instance, \cite{balogh}, \cite{CDPT}, \cite{vari}, \cite{DanGarN8,
gar}, \cite{FSSC3, FSSC5} , \cite{LeoM}, \cite{Monti}, \cite{Mag},
\cite{Monteb}, \cite{RR}, but the  list is far from being exhaustive. For a general overview of
sub-Riemannian (or Carnot-Charath\'eodory)  geometries, we refer the reader to  Gromov, \cite{Gr1},
Pansu, \cite{P4}, and Montgomery, \cite{Montgomery}.

In this paper, our ambient space is the {\it Heisenberg
group} $\mathbb{H}^n,\,n\geq 1,$ which can be regarded as $\mathbb{C}^n\times \R$ endowed with a
polynomial group law $\star:\mathbb{H}^n\times \mathbb{H}^n\longrightarrow\mathbb{H}^n$. Its Lie algebra $\mathfrak{h}_n$
identifies with the tangent space  $\TT_0\mathbb{H}^n$ at the
identity $0\in\mathbb{H}^n$.  Later on, $(z, t)\in\R^{2n+1}$ will denote exponential coordinates of a
generic point $p\in\mathbb H^n$. Now, take a left-invariant frame
$\mathcal{F}=\{X_1, Y_1,...,X_n, Y_n, T\}$ for the tangent bundle $\TT\mathbb{H}^n$,
where $X_i=\frac{\partial}{\partial x_i} -
\frac{y_i}{2}\frac{\partial}{\partial t}$,
$Y_i=\frac{\partial}{\partial y_i} +
\frac{x_i}{2}\frac{\partial}{\partial t}$  and
$T=\frac{\partial}{\partial t}$. Denoting by $[\cdot, \cdot]$ the
usual Lie bracket of vector fields, one has
$[X_i,Y_i]=T$ for every $i=1,...,n$ and all other commutators
vanish. Hence, $T$ is the {\it center} of $\mathfrak{h}_n$ and
 $\mathfrak{h}_n$ turns out to be nilpotent and stratified of step 2, i.e.
$\mathfrak{h}_n=\HH\oplus \HH_2$ where $\HH:={\rm span}_{\R}\{X_1,
Y_1,...,X_i,Y_i,...,X_n,Y_n\}$ is the  {\it horizontal bundle} and
$\HH_2={\rm span}_{\R}\{T\}$ is the $1$-dimensional vertical
bundle associated with the center of $\mathfrak{h}_n$. From now
on, $\mathbb{H}^n$ will be endowed with the (left-invariant)
Riemannian metric $h:=\langle\cdot, \cdot\rangle$ which makes
$\mathcal{F}$ an orthonormal frame. In particular, this metric
induces a corresponding metric $h\cc$ on $\HH$ which is used in
order to measure the length of horizontal curves. Note that the
natural distance in sub-Riemannian geometry is the  {\it
Carnot-Carath\'eodory  distance}  $\dc$, defined  by minimizing
the (Riemannian) length of all piecewise smooth horizontal curves
joining two different points. This definition makes sense because,
in view of Chow's Theorem, different points can be joined by
(infinitely many) horizontal curves.

The stratification of  $\mathfrak{h}_n$ is related with the
existence of a 1-parameter group of automorphisms, called {\it
Heisenberg dilations}, defined by $\delta_s (z, t):=(s z, s^2 t)$,
for every $p\equiv(z, t)\in\R^{2n+1}$. The intrinsic dilations play an important role in
this geometry. In this regard, we stress that the integer $Q=2n+2$, that is the \textquotedblleft homogeneous dimension\textquotedblright of $\mathbb H^n$ with respect to these anisotropic dilations, turns out to be the  dimension of $\mathbb{H}^n$ as a
metric space with respect to the {\it CC-distance} $\dc$.

Let us define a key notion: that of $\HH$-perimeter\footnote{Since we deal with smooth boundaries, we do not define the $\HH$-perimeter from a variational point of view.}.  So let $S\subset\mathbb{H}^n$ be a smooth hypersurface and let $\nu$ the (Riemannian) unit normal along $S$. The {\it $\HH$-perimeter measure} $\perh$ is the $(Q-1)$-homogeneous measure, with respect to the intrinsic dilations, given by
$\perh\res S:=|\PH\nu|\,\sigma^{2n}\rr$, where $\PH:\TG\longrightarrow\HH$ is the orthogonal projection operator onto $\HH$ and $\sigma^{2n}\rr$ denotes the Riemannian measure on $S$. The $\HH$-perimeter is in fact the natural measure on hypersurfaces and it turns out to be equivalent, up to a density function called metric factor (see, for instance, \cite{Mag}), to the spherical $(Q-1)$-dimensional Hausdorff measure associated with $\dc$ (or to any other homogeneous distance on $\mathbb{H}^n$).

Our main interest concerns  \textquotedblleft Isoperimetric Profiles\textquotedblright, that are compact
closed hypersurfaces which can be described in terms of
CC-geodesics (even if they are not CC-balls). In Heisenberg groups, they play an equivalent role of spheres in Euclidean spaces and for this reason it seems interesting to study some basic geometric features of these sets from an intrinsic point of view; see also \cite{Monted}.

Let us briefly describe them in the case of the 1st Heisenberg group
$\mathbb{H}^1$.  Any
CC-geodesic $\gamma\subset\mathbb{H}^1$ is either a Euclidean horizontal line or a \textquotedblleft suitable\textquotedblright infinite circular helix of
constant slope and whose axis is parallel to the center $T$  of $\mathfrak{h}_1$. In the last case, fix a point  ${\mathcal S} \in \gamma$ and
take the vertical $T$-line through $\mathcal S$. With no loss of generality, we may take $\mathcal S\in\left\lbrace (x, y, t)\in\R^3: x=y=0\right\rbrace$.  On this (positively oriented) line, there is a first consecutive point\footnote{This point
can be interpreted as the \textquotedblleft cut point\textquotedblright of $\mathcal S$ along $\gamma$. In fact this is the end-point of all CC-geodesics starting from $\mathcal S$ with same slope. Note however that, strictly speaking the cut locus of any point in $\mathbb{H}^n\, (n\geq 1)$ coincides with the vertical $T$-line over the point.} $\mathcal N$ to $\mathcal S$  belonging to $\gamma$.
These points, henceforth called
South and North poles,
determine a minimizing connected subset of $\gamma$.  Note that the slope of $\gamma$ is uniquely determined by the
CC-distance of the poles\footnote{Let $\gamma_0\subset\R^2$ denote the circle given by orthogonal
projection of $\gamma$ onto $\R^2$ and let $r$ be its radius. Then, it turns out that $\dc(\mathcal S, \mathcal N)=\pi r^2$.}.
Now rotating (around the vertical  $T$-line
joining $\mathcal{N}$ to $\mathcal{S}$) the connected subset of $\gamma$ joining the poles, yields a
closed convex surface very similar to an ellipsoid hereafter
called {\it Isoperimetric Profile}. A similar description holds even in the general case; see
Section \ref{Sez3}.

Isoperimetric Profiles, henceforth denoted by the symbol $\Sph$, turn out to be
constant horizontal mean curvature hypersurfaces (i.e. $\MS=-\div\cc\nn$ is constant; in particular, this implies that they are critical points of the $\HH$-perimeter functional) whose importance comes from a
long-standing conjecture, usually attributed to Pansu, claiming
that they minimize the $\HH$-perimeter in the class of finite
$\HH$-perimeter sets (in the variational sense) having fixed volume, or in other words, they
solve the  sub-Riemannian isoperimetric problem in $\mathbb{H}^n$. There is a wide
literature on this subject; see, for instance,  \cite{CDPT}, \cite{DanGarN8},
\cite{gar}, \cite{gar2}, \cite{LeoM}, \cite{Monti, Monti2},
\cite{Monti3}, \cite{Monti4},  \cite{P1, Pansu2}, \cite{RR} and references therein.\\

The plan  of the paper is the following.
In Section \ref{hngeo} we review the sub-Riemannian geometry of
Heisenberg groups $\mathbb{H}^n$. We then discuss some basics about smooth hypersurfaces endowed with the
$\HH$-perimeter measure $\perh$ and we prove some important geometric facts; see Section \ref{sez22}.
Section \ref{IBPAA} provides some
horizontal integration by parts formulas. In Section \ref{Sez3} we study Isoperimetric Profiles and
compute some of their geometric invariants appearing in the 2nd variation formula of the $\HH$-perimeter.  Section  \ref{1e2varfor} gives a self-contained account of
variational formulas for the $\HH$-perimeter
measure $\perh$ along the lines of \cite{Monteb}, but in addition we
consider the case of non-empty characteristic sets. These formulas are then used as a tool to study
the (homogeneous) sub-Riemannian Isoperimetric Functional\[J (D):=\frac{\perh(\partial D)}
{\left(\sigma^{2n+1}\rr(D)\right)^{1-\frac{1}{Q}}},\]where $D\subsetneq\mathbb H^n$ varies among $\cont^2$-smooth compact domains.  In Section \ref{volvarIS}, we  calculate  1st and 2nd variation of the top-dimensional volume form $\sigma^{2n+1}\rr$. This allow us to state the notion of {\it stability} for smooth domains bounded by constant horizontal mean
curvature hypersurfaces; see Definition \ref{strdef} and Definition \ref{strdef2}.
If $D$ has radial symmetry with respect to a  barycentric vertical axis, we also consider a restricted family of (normal) radial variations. In this case, the functional $J(D)$ becomes 1-dimensional and \it stability \rm becomes \it radial stability; \rm see Remark \ref{radstrdef}. We also introduce a localized notion of stability. Roughly speaking, being \it locally stable \rm means that for each point $p\in\partial D$ there exists a  neighborhood of $p$ which is stable in the previous sense; see Definition \ref{strdefloc}.  For  completeness, in the Appendix A we shall discuss the simpler (but less general) case of $T$-graphs. Moreover, in the Appendix B, we shall discuss some further properties which are related to stability.

Then, in
Section \ref{STABs} we begin the study of the
stability of Isoperimetric Profiles, or the positivity of the 2nd variation  of the isoperimetric functional $J(\cdot)$. Our approach was somehow motivated by the Riemannian case described here below; see, for instance, \cite{DC}.

We recall that the 2nd variation (under normal variations)
of a compact closed bounding hypersurface $S$ embedded in  Euclidean space $\R^n$ is provided by the formula
\[II(\varphi\nu, S)=\int_S\left( -\varphi\Delta\ts\varphi-\varphi^2\|B\|^2\ngr\right)\, \sigma^{n-1}\rr\]for every (piecewise) smooth $\varphi:S\longrightarrow\R$, where $\Delta\ts$ denotes the Laplace-Beltrami operator and $\|B\|\ngr$ is the Gram norm of the 2nd fundamental form $B$ of $S$. So let $\mathbb S^{n-1}\subset\Rn$ be the unit sphere and let us apply the Rayleigh   principle; see \cite{Ch1, Ch3}. We have \[\lambda_1:=\lambda_1(\mathbb S^{n-1})\leq\frac{\int_{\mathbb S^{n-1}}|\grad\ts\varphi|^2\, \sigma^{n-1}\rr}{\int_{\mathbb S^{n-1}}\varphi^2\, \sigma^{n-1}\rr}\] for every smooth function $\varphi:\mathbb S^{n-1}\longrightarrow\R$ such that $\int_{\mathbb S^{n-1}}\varphi\, \sigma^{n-1}\rr=0$, where $\lambda_1$ denotes the first non-trivial eigenvalue of (the closed eigenvalue problem on) $\mathbb S^{n-1}$. It is well-known that $\lambda_1=n-1$ and that $\|B\|^2\ngr=n-1$.  Therefore,
 \begin{eqnarray*}
 II(\varphi\nu, {\mathbb S^{n-1}})=\int_{\mathbb S^{n-1}}\left( -\varphi\Delta\ts\varphi-\varphi^2\|B\|^2\ngr\right)\, \sigma^{n-1}\rr=\int_{\mathbb S^{n-1}}\left( |\grad\ts\varphi|^2\, -(n-1)\varphi^2\right)\, \sigma^{n-1}\rr
\geq0,\end{eqnarray*}where we have used the  Divergence Theorem. This proves the stability of ${\mathbb S^{n-1}}$, i.e.  $II(\varphi\nu, {\mathbb S^{n-1}})\geq 0$ for every  differentiable function $\varphi:\mathbb S^{n-1}\longrightarrow\R$ such that $\int_{\mathbb S^{n-1}}\varphi\, \sigma^{n-1}\rr=0$.

However, a such strategy does not work verbatim in the framework of Heisenberg groups and our methods, although similar in spirit,  are very different.
Actually, the main analogy here is that the positivity of the 2nd variation formula can be studied in terms of an eigenvalue equation associated with the (2nd variation) functional$$\mathfrak{F}(\varphi):=\int_{\Sph}\left(|\qq\varphi|^2-
\frac{Q-4}{\rho^2}\varphi^2\right)\,\perh $$subject to the condition  $\int_{\Sph}\varphi\,\perh=0$. In this formula, $\qq$ denotes the horizontal tangent gradient operator and $\rho$ stands for the (Euclidean) distance from the vertical $T$-line passing through the barycenter of $\Sph$;  for a detailed discussion, see Section \ref{STABs}.\\

Our main results concerning stability of Isoperimetric Profiles can be summarized as follows: \it
\begin{itemize}\item Let $n=1$. The Isoperimetric profile $\Spu$ is a  stable bounding hypersurface in the sense of both Definition \ref{strdef} and  Definition \ref{strdef2}. \item Let $n>1$. The Isoperimetric profile $\Sph$ is a  radially stable bounding hypersurface in the sense of Remark \ref{radstrdef}. Furthermore, the Isoperimetric Profile $\Sph$ turns out to be a locally stable bounding hypersurface in the sense of Definition \ref{strdefloc}.\end{itemize}

\rm

This paper is part of a project aiming to study constant and minimal horizontal mean curvature
hypersurfaces, in the setting of
Heisenberg groups; see also \cite{Monted}. I would like to express my gratitude to Prof. N. Garofalo and to Prof. A. Parmeggiani for many interesting conversations about these topics over the past few years.

\subsection{Heisenberg group $\mathbb{H}^n$}\label{hngeo}

The {\it $n$-th Heisenberg group} $(\mathbb{H}^n,\star)$, $n\geq
1$, is a connected, simply connected, nilpotent and stratified Lie
group of step 2 on $\R^{2n+1}$, with respect to a polynomial group
law $\star$. The {\it Lie algebra} $\mathfrak{h}_n$ of
$\mathbb{H}^n$ is a $(2n+1)$-dimensional real vector space
henceforth identified with the tangent space  $\TT_0\mathbb{H}^n$ at
the identity $0\in\mathbb{H}^n$. We adopt {\it exponential
coordinates of the 1st kind} in such a way that every point
$p\in\mathbb{H}^n$  can be written out as
$p=\exp(x_1,y_1,...,x_i,y_i,...,x_n,y_n, t)$. The Lie algebra
$\mathfrak{h}_n$ can  be described by means of a frame
${\mathcal{F}}:=\{X_1,Y_1,...,X_i,Y_i,...,X_n,Y_n,T\}$
of left-invariant vector fields for $\TT\mathbb{H}^n$, where
$X_i(p):=\frac{\partial}{\partial x_i} -
\frac{y_i}{2}\frac{\partial}{\partial t},\,
Y_i(p):=\frac{\partial}{\partial y_i} +
\frac{x_i}{2}\frac{\partial}{\partial t},\,\, i=1,...,n,\,
T(p):=\frac{\partial}{\partial t},$ for every $p\in\mathbb{H}^n$.
More precisely, denoting by $[\cdot, \cdot]$ the Lie bracket of vector
fields, we get that $[X_i,Y_i]=T$ {for every} $i=1,...,n$, and all other commutators
vanish. In other words, $T$ is the {\it center} of
$\mathfrak{h}_n$ and
 $\mathfrak{h}_n$ turns out to be a nilpotent and stratified Lie algebra of
step 2, i.e. $\mathfrak{h}_n=\HH\oplus \HH_2$. The first layer $\HH$
is called {\it horizontal} whereas the complementary layer $\HH_2={\rm span}_{\R}\{T\}$
is called {\it vertical}.  A horizontal left-invariant frame for
$\HH$ is given by ${\mathcal{F}}\cc=\{X_1,
Y_1,...,X_i,Y_i,...,X_n,Y_n\}.$ The group law $\star$ on
$\mathbb{H}^n$ is determined by a corresponding operation
$\diamond$ on  $\mathfrak{h}_n$, i.e. $\exp X\star\exp
Y=\exp(X\diamond Y)$ for every $X,\,Y \in \mathfrak{h}_n,$ where
$\diamond:\mathfrak{h}_n\times \mathfrak{h}_n\longrightarrow
\mathfrak{h}_n$ is defined by $X\diamond Y= X + Y+
\frac{1}{2}[X,Y]$. Thus, for every
 $p=\exp(x_1,y_1,...,x_n,y_n,
t),\,\,p'=\exp(x'_1,y'_1,...,x'_n,y'_n, t')\in \mathbb{H}^n$ we
have
\[p\star p':= \exp\left(x_1+x_1', y_1+y_1',...,x_n+x_n', y_n+y_n', t+t'+ \frac{1}{2}\sum_{i=1}^n
\left(x_i y'_{i}- x'_{i} y_i\right)\right).\] The
inverse of ${p}\in\mathbb{H}^n$  is given by
${p}^{-1}:=\exp(-{x}_1,-y_1...,-{x}_{n}, -y_n, -t)$ and
$0=\exp(0_{\R^{2n+1}})$. Later on, we shall set $z:=(x_1, y_1,..., x_n, y_n)\in\R^{2n}$ and identify each point $p\in\mathbb H^n$ with its exponential coordinates $(z, t)\in\R^{2n+1}$.

\begin{Defi}\label{dccar} We call {\rm sub-Riemannian metric} $\g$ any
symmetric positive bilinear form on  $\HH$.
The {\rm {CC}-distance} $\dc(p,p')$ between $p, p'\in
\mathbb{H}^n$ is defined by
$$\dc(p, p'):=\inf \int\sqrt{\g(\dot{\gamma},\dot{\gamma})} dt,$$
where the $\inf$ is taken over all piecewise-smooth horizontal
curves $\gamma$ joining $p$ to $p'$. We shall equip
$\TT\mathbb{H}^n$ with the left-invariant Riemannian metric
  $h:=\langle\cdot,\cdot\rangle$ making ${\mathcal{F}}$ an
orthonormal -abbreviated o.n.- frame and assume $\g:=h|_\HH.$
\end{Defi}

By Chow's Theorem it turns out that every couple of points can
be connected by a horizontal curve, not necessarily unique, and for this reason $\dc$ turns out to be a metric on $\mathbb{H}^n$. Moreover, the $\dc$-topology is equivalent to the Euclidean
topology on $\R^{2n+1}$; see \cite{Gr1}, \cite{Montgomery}. The
so-called {\it structural constants} (see \cite{Helgason},
\cite{3} or \cite{Monte, Monteb}) of $\mathfrak{h}_n$ are
described by the skew-symmetric $(2n\times 2n)$-matrix
 $$C\cc^{2n+1}:=\left|
\begin{array}{ccccccc}
  0 & 1 & 0 & 0 &  \cdot &  0 &  0 \\
  -1 & 0 & 0 & 0 &   \cdot &  0 &  0 \\
  0 & 0 & 0 & 1 &   \cdot & 0 & 0 \\
  0 & 0 & -1 & 0 &  \cdot & 0 & 0 \\
  \cdot & \cdot & \cdot & \cdot & \cdot & \cdot & \cdot \\
  0 & 0 & 0 & 0  & \cdot & 0 & 1 \\
  0 & 0 & 0 & 0  & \cdot & -1 & 0
\end{array}%
\right|,$$which is the matrix associated
with the skew-symmetric bilinear map
$\Gamma\cc:\HH\times\HH\longrightarrow \R$ given by
$\Gamma\cc(X, Y)=\langle[X, Y], T\rangle$.

\begin{no}We set $z^\perp:=-C^{2n+1}\cc z=(-y_1, x_1, ...,-y_n,
x_n)\in\R^{2n}$ and $X^\perp:=-C^{2n+1}\cc X$ for every $X\in\HH$.
\end{no}
Given  $p\in\mathbb{H}^n$, we shall denote by
$L_p:\mathbb{H}^n\longrightarrow\mathbb{H}^n$ the {\it left
translation by $p$}, i.e. $L_pp'=p\star p'$, for every
$p'\in\mathbb{H}^n$. $L_p$ is a group homomorphism and its
differential
${L_p}_\ast:\TT_0\mathbb{H}^n\longrightarrow\TT_p\mathbb{H}^n$  is
given by the matrix
\begin{displaymath}
{L_p}_\ast=\frac{\partial (p\star p')}{\partial p'}\bigg|_{p'=0}=
\left[%
\begin{array}{ccccccccc}
  1 & 0 & \ldots &0&0& \ldots& 0 & 0 & 0 \\
  0 & 1 & \ldots &0&0& \ldots & 0 &0 & 0 \\
  \vdots & \vdots & \ldots & \vdots & \vdots & \ldots & \vdots & \vdots&\vdots \\
  0 & 0 & \ldots & 0& 0 & \ldots & 0 & 1 & 0 \\
 - \frac{y_1}{2} & +\frac{x_1}{2} & \ldots&-\frac{y_i}{2} & +\frac{x_i}{2} & \ldots & -\frac{y_n}{2} & +\frac{x_n}{2} & 1\\
\end{array}%
\right].
\end{displaymath}Equivalently, one has
${L_p}_\ast={\rm{col}}[X_1(p),Y_1(p),...X_n(p),Y_n(p),T(p)]$.\\
\indent There exists a 1-parameter group of automorphisms
$\delta_s:\mathbb{H}^n \longrightarrow\mathbb{H}^n\,(s\geq 0)$,
called {\it Heisenberg dilations}, defined by $\delta_s p
:=\exp\left(s z, s^2 t\right)$ for every $s\geq 0$, where $p=\exp(z,
t)\in\mathbb{H}^n$. We recall that the {\it homogeneous dimension}
of $\mathbb{H}^n$ is the integer $Q:=2n+2$. By a well-known result of Mitchell (see, for instance, \cite{Montgomery}), this number coincides with
the {\it Hausdorff dimension} of $\mathbb{H}^n$ as metric space
with respect to the CC-distance $\dc$; see, for instance, \cite{Gr1}, \cite{Montgomery}.\\
\indent We shall denote by $\nabla$ the unique {\it left-invariant
Levi-Civita connection} on $\TT\mathbb{H}^n$ associated with the metric
$h=\langle\cdot,\cdot\rangle$. We observe that, for every $X, Y, Z\in
\XX:=\cin(\mathbb{H}^n, \TT\mathbb{H}^n)$ one has
\[\left\langle\nabla_XY,Z\right\rangle=\frac{1}{2}
\left(\langle[X, Y], Z\rangle-\langle[Y, Z], X\rangle + \langle[Z,
X], Y\rangle\right).\]\indent For every $X,
Y\in\XX\cc:=\cin(\mathbb{H}^n,\HH)$, we shall set $\gc_X
Y:=\PH\left(\nabla_X Y\right),$ where $\PH$ denotes the orthogonal
projection operator onto $\HH$. The operation $\gc$ is a vector-bundle connection later called
{\it $\HH$-connection}; see \cite{Monteb} and references therein.
It is not difficult to see that $\gc$ is {\it flat}, {\it
compatible with the sub-Riemannian metric} $h\cc$ and {\it
torsion-free}. These properties follow from the very
definition of $\gc$ and from the corresponding properties of the
Levi-Civita connection $\nabla$.

\begin{Defi}
For any $\psi\in\cin(\mathbb{H}^n)$, the {\rm $\HH$-gradient of
$\psi$} is the horizontal vector field $\dg \psi\in\XX\cc$
such that $\langle\dg \psi,X \rangle= d \psi (X) = X \psi$ for
every $X\in \HH$. The {\rm $\HH$-divergence} $\divh X$ of
$X\in\XX\cc$ is defined, at each point $p\in \mathbb{H}^n$, by
$$\divh X(p):= \mathrm{Trace}\left(Y\longrightarrow \gc_{Y} X
\right)(p)\quad(Y\in \HH_p).$$The {\rm $\HH$-Laplacian}
$\Delta\cc$  is the 2nd order differential operator given by
\[
\Delta\cc\psi := \div\cc(\dg\psi)\quad\mbox{for every}\,\,\psi\in
\cin(\mathbb{H}^n).\]
\end{Defi}

Having fixed a left-invariant Riemannian metric $h$ on
$\TT\mathbb{H}^n$, one  defines by duality\footnote{The duality is understood with respect to the left-invariant metric $h$.} a global coframe
${\mathcal{F}}^\ast:=\{X^\ast_1,Y^\ast_1,...,X^\ast_i,Y^\ast_i,...,X^\ast_n,Y^\ast_n,T^\ast\}$
of  left-invariant $1$-forms for the cotangent bundle
$\TT^\ast\mathbb{H}^n$, where $X_i^\ast= dx_i,\,
Y_i^\ast=dy_i\,\,(i=1,...,n)$ and $$\theta:=T^\ast=dt +
\frac{1}{2}\sum_{i=1}^n\left(y_i d x_i  -  x_i d y_i\right).$$ The
differential $1$-form $\theta$ is called {\it contact form} of
$\mathbb{H}^n$.   The {\it Riemannian left-invariant volume form}
$\sigma^{2n+1}\rr\in\bigwedge^{2n+1}(\TT^\ast\mathbb{H}^{n})$ is
given by $\sigma^{2n+1}\rr:=\left(\bigwedge_{i=1}^n dx_i \wedge d
y_i\right) \wedge \theta$ and the measure obtained by integrating
$\sigma^{2n+1}\rr$ is the {\it Haar measure} of $\mathbb{H}^{n}$.

\subsection{Hypersurfaces  and some geometric
calculations}\label{sez22}

 Let $S\subset\mathbb{H}^n$ be a
$\cont^1$-smooth hypersurface and let $\nu$ be the (Riemannian)
unit normal along $S$. Remind that the Riemannian measure
$\sigma^{2n}\rr\in\bigwedge^{2n}(\TT^\ast S)$ on hypersurfaces can
be defined by {\it contraction}\footnote{Let $M$ be a Riemannian
manifold. The linear map $\LL: \Lambda^r(\TT^\ast
M)\rightarrow\Lambda^{r-1}(\TT^\ast M)$ is defined, for $X\in \TT
M$ and $\omega^r\in\Lambda^r(\TT^\ast M)$, by $(X \LL \omega^r)
(Y_1,...,Y_{r-1}):=\omega^r (X,Y_1,...,Y_{r-1})$; see, for
instance, \cite{FE}. This operation is called {\it contraction} or
{\it interior product}.} of the top-dimensional volume form
$\sigma^{2n+1}\rr$ with the unit normal $\nu$ along $S$, i.e.
$\sigma^{2n}\rr\res S := (\nu\LL \sigma^{2n+1}\rr)|_S$.

We say that $p\in S$ is a {\it characteristic point}  if
$\dim\,\HH_p = \dim (\HH_p \cap \TT_p S)$. The {\it characteristic
set} of $S$ is the set of all characteristic points, i.e. $
C_S:=\{x\in S : \dim\,\HH_p = \dim (\HH_p \cap \TT_p S)\}.$ It is
worth noticing that $p\in C_S$ if,
 and only if, $|\PH\nu(p)|=0$. Since $|\PH\nu(p)|$ is continuous
 along $S$, it follows that
$C_S$ is a closed subset of $S$, in the relative topology. We
stress that characteristic points are few. More precisely, under our current assumptions the
$(Q-1)$-dimensional Hausdorff measure of $C_S$ vanishes, i.e.
$\mathcal{H}_{CC}^{Q-1}(C_S)=0$; see \cite{balogh}, \cite{Mag}.
\begin{oss}\label{CSET}Let $S\subset \mathbb{H}^n$ be a $\cont^2$-smooth hypersurface.
By using {\rm Frobenius' Theorem}
 about integrable distributions, it can be shown that the topological dimension
 of $C_S$ is strictly less than $(n+1)$; see also \cite{Gr1}. For deeper results  about the size of $C_S$
  in $\mathbb{H}^n$, see \cite{balogh}, \cite{Bal3}.\end{oss}
Throughout this paper we make use of a homogeneous measure on
hypersurfaces, called {\it $\HH$-perimeter measure}; see also
\cite{FSSC3}, \cite{G}, \cite{DanGarN8, gar}, \cite{Mag},
\cite{Monte, Monteb}, \cite{P1},  \cite{RR}.

\begin{Defi}[$\perh$-measure]\label{sh}
 Let $S\subset\mathbb{H}^n$ be a $\mathbf{C}^1$-smooth
non-characteristic
 hypersurface and let $\nu$ be the unit
normal vector along $S$. The  {\rm unit $\HH$-normal} along $S$ is defined by $\nn:
=\frac{\PH\nu}{|\PH\nu|}.$ Then, the {\rm $\HH$-perimeter form}
$\perh\in\bigwedge^{2n}(\TT^\ast S)$ is the contraction of the
volume form $\sigma^{2n+1}\rr$ of $\mathbb{H}^n$ by the
horizontal unit normal $\nn$, i.e. $$\perh \res S:=\left(\nn \LL
\sigma^{2n+1}\rr\right)\big|_S.$$\end{Defi} If
$C_S\neq\emptyset$ we  extend $\perh$ up to $C_S$  by
setting $\perh\res C_{S}= 0$. It turns out that $\perh \res S =
|\PH \nu |\,\sigma^{2n}\rr\, \res S$.

At each non-characteristic point $p\in S\setminus {C}_S$ one has $\HH_p= {\rm
span}_\R\{\nn(p)\} \oplus \mathit{H}_p S$, where  $\mathit{H}_p
S:=\HH_p\cap\TT_p S$. This allow us to define, in the obvious way, the
associated subbundles $\HS \subset \TS$ and $\nn S$  called {\it horizontal tangent bundle} and {\it horizontal
normal bundle} along $S\setminus C_S$,
respectively.  On the other hand, at each  characteristic point $p\in C_S$,  only the subbundle $\HS$ turns out to be defined, and in this case  $\HH_pS=\HH_p$. Another important geometric object  is given by
$\varpi:=\frac{\nu_{T}}{|\PH\nu|}$; see \cite{Monte, Monteb},
\cite{gar}. Although the function $\varpi$ is not defined at $C_S$, we have $\varpi\in L^1_{loc}(S,
\perh)$.

\begin{no}Let $S\subset\mathbb{H}^n$ be a $\mathbf{C}^k$-smooth
hypersurface. We shall denote by
$\cont^i\ss(S),\,i=1, 2,...,k$, the space of functions whose $i$-th
$\HS$-derivatives are continuous\footnote{We are requiring that all  $i$-th
$\HS$-derivatives be continuous at each characteristic point $p\in C_S$.}.  An analogous notation will be used for open subsets of $S$.
\end{no}

The following definitions can also be found in \cite{Monteb}, for
general Carnot groups. {\it Below, unless otherwise specified, we
shall assume that $S\subset\mathbb{H}^n$ is a $\cont^2$-smooth
non-characteristic hypersurface}. Let $\tsc$ be the connection on
$S$ induced from the Levi-Civita connection $\nabla$ on
$\mathbb{H}^n$. As for the horizontal connection $\gc$, we define
a \textquotedblleft partial connection\textquotedblright $\gs$
associated with the subbundle $\HS\subset\TT S$ by setting
$$\gs_XY:=\P\ss\left(\tsc_XY\right)$$ for every
$X,Y\in\XX^1\ss:=\cont^1(S, \HS)$, where
$\P\ss:\TT{S}\longrightarrow\HS$ denotes the orthogonal projection
operator of $\TT{S}$ onto $\HS$. Starting from the orthogonal
splitting
 $\HH=\nn S\oplus\HS$,
it can be shown that$$\gs_XY=\gc_X Y-\left\langle\gc_X
Y,\nn\right\rangle \nn\quad\mbox{for every}\,\,X,
Y\in\XX^1\ss.$$\begin{Defi}\label{curvmed}Given $\psi\in \cont\ss^1(S)$, we define the
$\HS$-{\rm gradient of $\psi$} to be the horizontal tangent
vector field  $\qq\psi\in\XX^0\ss:=\cont(S, \HS)$ such that $\langle\qq\psi,X
\rangle= d \psi (X) = X \psi$ for every $X\in \HS$. The {\rm
$\HS$-divergence} $\div\ss X$ of $X\in\XX^1\ss$ is given, at each
point $p\in S$, by
$$\div\ss X (p) := \mathrm{Trace}\left(Y\longrightarrow
\gs_Y X \right)(p)\quad\,(Y\in \HH_pS).$$Note that $\div\ss X \in\cont(S)$. The {\rm $\HS$-Laplacian}
$\Delta_{_{\HS}}:\cont\ss^2(S)\longrightarrow\cont(S)$ is the 2nd order differential operator given by
\[\Delta\ss\psi := \div\ss(\qq\psi)\qquad\mbox{for every}\,\,\,\psi\in
\cont\ss^2({S}).\]
The horizontal 2nd fundamental form of $S$ is the bilinear map  ${B\cc}:
\XX^1\ss\times\XX^1\ss \longrightarrow C(S)$ defined by
\[B\cc(X,Y):=\left\langle\gc_X Y,\nn\right\rangle\qquad \mbox{for every}\,\,
X,\,Y\in\XX^1\ss.\]The {\rm horizontal mean curvature} is  the trace
of $B\cc$, i.e. $\MS:=\mathrm{Tr}{B\cc}$. We shall set\begin{equation*}
W\cc X:=\gc_X\nn\qquad \mbox{for every}\,\,X\in\XX\ss^1.\end{equation*} The {\rm torsion}
$\Tor\ss$ of  $\gs$ is given by
$\Tor\ss(X,Y):=\gs_XY-\gs_YX-\PH[X,Y]$ for every $X, Y\in\XX^1\ss$.
\end{Defi}

If $n=1$, the horizontal tangent space $\HS$ is 1-dimensional
and the torsion vanishes, but if $n>1$ this is no longer true in
general, because $B\cc$ is {\it not symmetric}; see \cite{Monteb}.
Therefore, it is convenient to represent $B\cc$ as a
sum of two operators, one symmetric and the other skew-symmetric,
i.e. $B\cc= S\cc + A\cc$. It turns out
that
$A\cc=\frac{1}{2}\varpi\,C^{2n+1}\cc\big|_{\HS}$; see
 \cite{Monteb}. The linear operator
$C^{2n+1}\cc$ only  acts on horizontal tangent vectors and hence  we shall set
$C^{2n+1}\ss:=C^{2n+1}\cc\big|_{\HS}$.

\begin{Defi} In analogy with the
Riemannian case, the eigenvalues $\kappa_i , i\in I\ss,$
of the symmetric linear map $S\cc$ are called {\rm principal horizontal
curvatures}.
\end{Defi}

\begin{Defi}\label{movadafr}Let $S\subset\mathbb{H}^n$ be a $\cont^2$-smooth non-characteristic hypersurface.
We  call {\rm adapted  frame} along $S$  any o.n. frame
${\mathcal{F}}:=\{\tau_1,...,\tau_{2n+1}\}$ for
$\TT\mathbb{H}^n$
 such that:
\[\tau_1|_S=\nn,\qquad \HH_pS=\mathrm{span}_\R\{\tau_2(p),...,\tau_{2n}(p)\}\quad\mbox{for every}\,\,\,p\in S,
\qquad\tau_{2n+1}:= T.\]Furthermore, we shall set
$I\cc:=\{1,2,3,...,2n\}$ and
 $I\ss:=\{2,3,...,2n\}$.\end{Defi}

\begin{lemma}\label{Sple} Let $S\subset\mathbb{H}^n$  be a $\cont^2$-smooth
non-characteristic hypersurface and fix $p\in S$.
We can always choose an adapted o.n. frame
${\mathcal{F}}=\{\tau_1,...,\tau_{2n+1}\}$ along $S$
such that $\langle\nabla_{X}{\tau_i}, \tau_j\rangle=0$ at $p$ for
every $i, j\in I\ss$ and every $X\in\HH_{p}S$.
\end{lemma}
For a proof, see Lemma 3.8. in \cite{Monteb}. We end this section
by stating some useful technical lemmata. In the next proofs we shall make  use of an adapted o.n. frame ${\mathcal F}$ along $S$.

\begin{lemma}\label{Doca}Under the previous assumptions, let us further suppose
 that $S$ has constant horizontal mean
curvature $\MS$. Then
\begin{equation*}\|B\cc\|^2\ngr=-\sum_{i\in
I\ss}\left\langle\gc_{\tau_i}\gc_{\tau_i}\nn,\nn\right\rangle,
\end{equation*}where $\|\cdot\|\ngr$ is the  \textquotedblleft Gram norm of a linear operator\textquotedblright; see \cite{Ch1}.\end{lemma}
\begin{proof}Since
 $\langle\nn,\nn\rangle=1$ we get that $\left\langle\gc_{\tau_i}\nn,\nn\right\rangle=0$ for every $i\in I\ss$.
 Therefore
\begin{eqnarray*}\sum_{i\in
I\ss}\left\langle\gc_{\tau_i}\gc_{\tau_i}\nn,\nn\right\rangle &=&
- \sum_{i\in
I\ss}\left\langle\gc_{\tau_i}\nn,\gc_{\tau_i}\nn\right\rangle\\&=&-
\sum_{i,j,k\in
I\ss}\left\langle\gc_{\tau_i}\nn,{\tau_j}\right\rangle\left\langle\gc_{\tau_i}\nn,{\tau_k}\right\rangle\left\langle\tau_j,\tau_k\right\rangle\\&=&-
\sum_{i,j\in
I\ss}\left\langle\gc_{\tau_i}\nn,\tau_j\right\rangle^2\\&=&-\|B\cc\|^2\ngr.\end{eqnarray*}\noindent
\end{proof}

\begin{lemma}\label{fios}Under our previous assumptions, we have:\begin{itemize}\item[{\rm(i)}]
$\mathrm{Tr}\big(B\cc(\,\cdot\,, A\cc\,
 \cdot)\big)=\|A\cc\|^2\ngr=\frac{n-1}{2}\varpi^2;$\item[{\rm(ii)}]$\mathrm{Tr}\big(B\cc(\,\cdot\,, C^{2n+1}\ss\,
 \cdot)\big)=(n-1)\varpi$.\end{itemize}
\end{lemma}
\begin{proof}We claim that $\mathrm{Tr}\big(S\cc(\,\cdot\,, A\cc\,
 \cdot)\big)=0$. In order to prove this identity we compute
 \begin{eqnarray*}\langle S\cc\tau_i,A\cc\tau_i\rangle&=&
 \frac{1}{4}\left\langle \left(B\cc+B^\ast\cc\right)\tau_i, \left(B\cc-B^\ast\cc\right)\tau_i\right\rangle\\&=& \frac{1}{4}\left\{\langle B\cc\tau_i, B\cc\tau_i\rangle -\langle B^\ast\cc\tau_i,
 B^\ast\cc\tau_i\rangle+\underbrace{\langle B\cc\tau_i, B^\ast\cc\tau_i\rangle -\langle B^\ast\cc\tau_i,
 B\cc\tau_i\rangle}_{=0}
\right\}\\
 &=& \frac{1}{4}\big\{\langle B^\ast\cc B\cc\tau_i, \tau_i\rangle -\langle B\cc B^\ast\cc\tau_i,
 \tau_i\rangle \big\}\\&=&\frac{1}{4}\big\{\langle B\cc\tau_i, B\cc\tau_i\rangle -\langle B^\ast\cc\tau_i,
 B^\ast\cc\tau_i\rangle \big\}\end{eqnarray*}for any $i\in
 I\ss$.  Summing up over $i\in I\ss$ yields
\[\mathrm{Tr}\big(S\cc(\,\cdot\,, A\cc\,
 \cdot)\big)=\|B\cc\|^2\ngr-\|B^\ast\cc\|^2\ngr=0\] and the claim follows.
Now let us compute
\begin{eqnarray*}\mathrm{Tr}\big(B\cc(\,\cdot\,, A\cc\,
 \cdot)\big)=\sum_{i\in I\ss}\langle (S\cc+ A\cc)\tau_i,
 A\cc\tau_i\rangle=\sum_{i\in I\ss}\langle A\cc\tau_i,
 A\cc\tau_i\rangle=\|A\cc\|^2\ngr
 =\frac{n-1}{2}\,\varpi^2.\end{eqnarray*}This proves (i).
Finally, (ii) follows from (i) by using the identity
$A\cc=\frac{1}{2} \varpi C^{2n+1}\ss$.\end{proof}

In the preceding proof we have used the identity $\|A\cc\|^2\ngr
 =\frac{n-1}{2}\,\varpi^2$; see \cite{Monteb}, Example 4.11, p.
 470. This identity can easily be proved by making use of
 an adapted o.n. frame  ${\mathcal F}$ along $S$.
 Furthermore, we observe that  $\nn^\perp\in {\rm Ker} A\cc$, where
$\nn^\perp=-C^{2n+1}\cc\nn$.

\begin{oss}\label{fgh}The following  holds$$\langle W\cc X, Y\rangle-\langle W\cc Y, X\rangle=-\varpi\langle C\ss^{2n+1}X,Y\rangle\qquad\mbox{for every}\,\, X, Y\in\XX^1\ss.$$The proof of this identity uses the fact that the bracket of tangent vector fields is tangent to $S$.\end{oss}

\begin{lemma}\label{fgh2}For every $ X, Y\in\XX^1\ss$ one has $S\cc(X, Y)=- \langle W\cc X,Y\rangle-\dfrac{\varpi}{2}\langle C\ss^{2n+1}X,Y\rangle$.\end{lemma}
\begin{proof}Let $X, Y\in\XX^1\ss$ and compute \begin{eqnarray*}S\cc(X, Y)&=&\dfrac{1}{2}\left(\langle\gc_XY,\nn\rangle+\langle\gc_YX,\nn\rangle\right)\qquad\mbox{(by definition of $S\cc$)}\\
&=&-\dfrac{1}{2}\left(\langle\gc_X\nn,Y\rangle+\langle\gc_Y\nn,X\rangle\right)\qquad\mbox{(compatibility of $\gc$ with the metric $\langle\cdot,\cdot\rangle$)}\\&=&-\dfrac{1}{2}\left(\langle W\cc X,Y\rangle+\langle W\cc Y,X\rangle\right)\\&=&-\dfrac{1}{2}\left(2\langle W\cc X,Y\rangle+\varpi\langle C\ss^{2n+1}X,Y\rangle\right)\qquad\mbox{(by Remark \ref{fgh})}\\&=&- \langle W\cc X,Y\rangle-\dfrac{\varpi}{2}\langle C\ss^{2n+1}X,Y\rangle.\end{eqnarray*}\end{proof}
\begin{oss}[Characteristic direction and CC-geodesics]\label{njn} Let $S$ be a smooth hypersurface and assume that there exists a  CC-geodesic  $\gamma:]-\epsilon, \epsilon[\longrightarrow\mathbb H^n$ such that $\gamma\subset S$ and $ \dfrac{d{\gamma}}{ds}=\nn^\perp(\gamma),\,s\in]-\epsilon, \epsilon[$; see also Remark \ref{ossgeoip}.
As a consequence $$\dfrac{d\nn^\perp}{ds} =-
\lambda C^{2n+1}\cc\nn^\perp=-\lambda \nn \qquad\mbox{for every $s\in]-\epsilon, \epsilon[$}$$for some constant $\lambda$. It follows that at each point of $\gamma\cap S$ one must have $\gc_{\nn^\perp}{\nn^\perp}=-\lambda \nn$ or,
in other words, this shows that $S\cc(\nn^\perp,\nn^\perp)=-\lambda$.\end{oss}

\subsection{Homogeneous measure on $\partial S$ and horizontal integration by parts}\label{IBPAA}

Let $S\subset\mathbb{H}^n$ be a $\cont^2$-smooth compact
hypersurface with boundary. Let
$\partial S$ be a $(2n-1)$-dimensional (piecewise)  $\cont^1$-smooth
manifold, oriented by its unit normal vector $\eta\in\TT S$, and denote by $\sigma^{2n-1}\rr$  the Riemannian measure on $\partial S$ defined by setting $\sigma^{2n-1}\rr\res{\partial S}=(\eta\LL\sigma^{2n}\rr)|_{\partial S}$. Note  that
$$(X\LL\perh)|_{\partial S}=\langle X, \eta\rangle
  |\PH\nu|\, \sigma^{2n-1}\rr\res{\partial S}$$for every $X\in \cont(S, \TS).$
The
  {\it
characteristic set} $C_{\partial S}$ is defined
as $C_{\partial S}:=\{p\in{\partial S}: |\PH\nu||\P\ss\eta|=0\}$.
The {\it unit $\HS$-normal} along $\partial S$ is given by
$\eta\ss:=\frac{\P\ss\eta}{|\P\ss\eta|}$. As for the
$\HH$-perimeter measure, we define a homogeneous measure
${\sigma^{2n-1}\ss}$ along $\partial S$ by setting
 $${{\sigma^{2n-1}\ss}}\res{\partial S}:=
 \left(\eta\ss\LL\perh\right)\big|_{\partial S}.$$
We have $\sigma^{2n-1}\ss\res{\partial S}=
|\PH\nu|\,|\P\ss\eta|\,\sigma^{2n-1}\rr\res{\partial S}$; furthermore
$(X\LL\perh)|_{\partial S}=\langle X, \eta\ss\rangle\,
  {\sigma^{2n-1}\ss}\res{\partial S}$ {for every} $X\in\XX^0\ss$.

\begin{Defi}[Horizontal tangential operators]\label{Deflh} For simplicity, let us assume that $S\subset \mathbb H^n$ is non-characteristic. Later on, we shall denote by $\lg:\XX^1\ss\longrightarrow\cont(S)$ be
the 1st order
differential operator given by
\begin{eqnarray*}\lg(X):=\div\ss X + \varpi\langle C^{2n+1}\cc\nn,
X\rangle=\div\ss X -\varpi\langle\nn^{\perp}, X\rangle\qquad
\mbox{for every}\,\,X\in\XX^1\ss.\end{eqnarray*}Moreover, we shall denote by $\lh:\cont\ss^2(S)\longrightarrow\cont(S)$  the 2nd order differential operator given by
\begin{eqnarray*}\lh\varphi:=\lg(\qq \varphi)=\Delta\ss\varphi -\varpi\frac{\partial\varphi}{\partial\nn^\perp}
\qquad\mbox{for every}\,\,\varphi\in\cont^2\ss(S).\end{eqnarray*}
\end{Defi}
Note that, in the characteristic case, the operators  $\lg$
and $\lh$ are not  defined at $C_S$. The next integral  formula was formerly
proved for non-characteristic hypersurfaces but it holds true even in case of non-empty
characteristic sets; see \cite{Monte,
Monteb} and \cite{Monted}.

\begin{teo}\label{GD}Let $S\subset\mathbb{H}^n$ be a $\cont^2$-smooth
compact hypersurface with piecewise $\cont^1$-smooth boundary
$\partial S$. If $n=1$ assume further that $C_S$ is contained in a
finite union of $\cont^1$-smooth horizontal curves. Then
\[\int_{S}\lg(X)\,\perh=-\int_{S}\MS\langle X, \nn\rangle\,\perh +
\int_{\partial S}\langle
X,\eta\ss\rangle\,{{\sigma^{2n-1}\ss}}\qquad\mbox{for
every}\,\,X\in\XX^1\cc=\cont^1(\mathbb H^n, \HH).\]\end{teo}

Note that, if $X\in\XX^1\ss$  the first integral on the right hand side vanishes and, in this case, the formula is referred as \textquotedblleft horizontal divergence formula\textquotedblright.
We collect below some useful Green's formulas for the operator $\lh$.

\begin{corollario}\label{Properties}Let $S\subset\mathbb{H}^n$ be a $\cont^2$-smooth compact
hypersurface with piecewise $\cont^1$-smooth boundary $\partial
S$. If $n=1$ assume further that $C_S$ is contained in a finite
union of $\cont^1$-smooth horizontal curves. Under the previous
notation, the following hold:\begin{itemize} \item[{\rm(i)}]
$\int_{S}\lh \varphi\,\perh=0$ for every compactly supported
$\varphi\in\cont^2\ss(S)$; \item[ {\rm(ii)}]$\int_{S}\lh
\varphi\,\perh=\int_{\partial
S}{\partial\varphi}/{\partial\eta\ss}\,\sigma^{2n-1}\ss$ for every
$\varphi\in\cont^2\ss(S)$;\item[{\rm(iii)}]$\int_{S}\psi\,\lh
\varphi\,\perh=\int_{S}\varphi\,\lh\psi \,\perh$  for every
compactly supported $\varphi,\,\psi\in\cont^2\ss(S)$;
\item[{\rm(iv)}]$\int_{S}(\psi\,\lh \varphi-\varphi\,\lh\psi)
\,\perh=\int_{\partial
S}(\psi{\partial\varphi}/{\partial\eta\ss}-\varphi{\partial\psi}/{\partial\eta\ss})
\,\sigma^{2n-1}\ss$ for every
$\varphi,\,\psi\in\cont^2\ss(S)$;\item[
{\rm(v)}]$\int_{S}\psi\,\lh
\varphi\,\perh=-\int_{S}\langle\qq\varphi,\qq\psi\rangle\,\perh +
\int_{\partial
S}\psi{\partial\varphi}/{\partial\eta\ss}\,\sigma^{2n-1}\ss$
 for every
$\varphi,\,\psi\in\cont^2\ss(S)$;\item[{\rm(vi)}]$\int_{S}\lh
(\varphi^2)\,\perh=2\int_{S}\varphi\lh
\varphi\,\perh+2\int_{S}|\qq\varphi|^2\,\perh =
\int_{\partial S}{\partial\varphi^2}/{\partial\eta\ss}\,\sigma^{2n-1}\ss$
for every $\varphi\in\cont^2\ss(S)$.
\end{itemize}
\end{corollario}
The proof of the characteristic case follows from the
non-characteristic one by dominated convergence  together with some elementary estimates. The starting point of the proof is to cover the
characteristic set by a family of subsets
$\{\UU_\epsilon\}_{\epsilon\geq 0}$ such that:
(i) ${\rm Car}_S\Subset\UU_\epsilon$ for every
$\epsilon>0$;\,\, (ii) $\sigma^{2n}\rr(\UU_\epsilon)\longrightarrow 0$
for $\epsilon\rightarrow 0^+$; \,\, (iii)
$\int_{\partial\UU_\epsilon}|\PH\nu|\,\sigma^{2n-1}\rr\longrightarrow
0$ for $\epsilon\rightarrow 0^+$; see also
\cite{Monted}. It is not difficult to see that, under the previous
assumptions, such a family does exist. Later on this idea will be
used in order to extend the variational
formulas for the $\HH$-perimeter measure proved in \cite{Monte, Monteb}, to characteristic hypersurfaces.

\begin{oss}A simple way to state Stokes formula is the following: \begin{enumerate}
 \item [] \rm Let $M$ be an oriented $k$-dimensional manifold of class $\cont^2$ with boundary
$\partial M$. Then $$\int_M d\alpha=\int_{\partial M}\alpha$$for every compactly supported $(k-1)$-form $\alpha$ of class $\cont^1$.
\end{enumerate}
Without much effort, it is possible to extend  this formula
to the case where:\begin{center}$(\star)$\quad \rm $M$ is of class $\cont^1$ and $\alpha$ is a $(k-1)$-form such that $\alpha$ and $d\alpha$ are continuous.

\end{center}\it

For a more detailed discussion see \cite{Taylor}.

\end{oss}

The  previous condition $(\star)$ can be used  to extend the previous formulas to vector fields (and functions)  possibly singular at the characteristic set $C_S$. So let $S\subset\mathbb{H}^n$
 be a $\cont^2$-smooth hypersurface  with (piecewise) $\cont^1$-smooth boundary $\partial S$ and let $X\in\cont^1(S\setminus C_S, \HS)$. Set $$\alpha_X:=(X\LL \perh)|_S.$$
Then, condition $(\star)$  requires that $\alpha_X$ and $d\alpha_X$ be continuous on $S$. Note that $X$ is of class $\cont^1$ out of $C_S$ but may be singular at $C_S$.
For later purposes, we also define the space of \textquotedblleft admissible\textquotedblright functions for the horizontal Green's
 formulas (iii)-(vi) of Corollary \ref{Properties}.
\begin{Defi}\label{fi} Let $X\in\cont^1(S\setminus C_S, \HS)$ and set
 $\alpha_X:=(X\LL \perh)|_S$. We say that $X$ is \rm admissible  (for the horizontal divergence formula)  \it if, and only if, the differential forms $\alpha_X$ and $d\alpha_X$ are continuous on all of $S$. We say that  $\varphi\in\cont^2\ss(S\setminus C_S)$  is \rm admissible (for the  horizontal Green's formulas (iii)-(vi) stated in Corollary \ref{Properties}) \it if, and only if, $\psi\,\qq \varphi$ is admissible (for the horizontal divergence formula)   for every $\psi\in\cont^2\ss(S\setminus C_S)$ such that  $\psi\,\qq \psi$ is admissible (for the horizontal divergence formula). We shall denote by $\varPhi(S)$ the  space of all admissible functions.
\end{Defi}

\section{Isoperimetric Profiles}\label{Sez3}

\begin{oss}[CC-geodesics and Isoperimetric Profiles]\label{ossgeoip} By definition, CC-geodesics are
horizontal curves which minimize the CC-distance. In Heisenberg
groups,  they are obtained by solving the following system of O.D.E.s:
\begin{eqnarray}\label{ngsh1}\left\{\begin{array}{ll}\dot{\gamma}=P\cc\\\dot{P}\cc=-
P_{2n+1}C^{2n+1}\cc
P\cc\\\dot{P}_{2n+1}=0.\end{array}\right.\end{eqnarray}Equivalently, the 2nd equation  to solve is given by
$\dot{P}\cc=P_{2n+1} P^\perp\cc$, where
$P\cc=(P_1,...,P_{2n})^{\rm Tr}$, $|P\cc|=1$. The quantity $P_{2n+1}$ turns
out to be a constant parameter along $\gamma$. The vector
$P=(P\cc, P_{2n+1})\in\R^{2n+1}$ can be regarded as a vector of ``Lagrangian
multipliers''. Solutions of \eqref{ngsh1} are called {\rm normal
CC-geodesics}. We stress that \eqref{ngsh1} can be deduced by
 minimizing the constrained Lagrangian
$\mathbf{L}(t,\gamma,\dot{\gamma})=|\dot{\gamma}\cc| +
P_{2n+1}\theta(\dot{\gamma})$; see \cite{Montegv} and references
therein, or \cite{Montgomery}. Unlike the Riemannian case,
CC-geodesics in $\mathbb{H}^n$ depend not only on the initial
point $\gamma(0)$ and on the initial direction $P\cc(0)$, but also on
the parameter $P_{2n+1}$. Now if $P_{2n+1}=0$, CC-geodesics are
Euclidean horizontal lines. Furthermore, if $P_{2n+1}\neq 0$,
any CC-geodesic turns out to be a
\textquotedblleft helix\textquotedblright\footnote{If $n=1$, $\gamma(t)$ is a circular helix with axis
parallel to the vertical direction $T$ and whose slope depends on
$P_3$. We stress that the projection of $\gamma(t)$ onto $\R^2$ turns
out to be a circle whose radius explicitly depends on $P_3$.}. To be more precise, the horizontal projection of any CC-geodesic  $\gamma$ onto
$\HH_0=\R^{2n}$ belongs to a sphere whose radius only depends on
$P_{2n+1}$. Now take a point  ${\mathcal S} \in \gamma$ and
consider the,  positively oriented, vertical $T$-line  over this point. On this line, there exists a first consecutive point\footnote{This point
is a sort of \textquotedblleft cut point\textquotedblright of $\mathcal S$ along $\gamma$. Actually, this is the end-point of all CC-geodesics starting from $\mathcal S$ with same slope. However the properly said cut locus of any point in $\mathbb{H}^n\, (n\geq 1)$ is the vertical $T$-line over that point.} $\mathcal N$ to $\mathcal S$
  belonging to $\gamma$. It can be proved that these two points, henceforth called
South and North poles,
determine a minimizing connected subset of $\gamma$.  By rotating this curve around the  $T$-axis
passing from $\mathcal S$ we obtain a
closed convex surface, which is the so-called \rm{Isoperimetric Profile.}
\end{oss}

Later on we shall  study some features of a model
Isoperimetric Profile, having barycenter at
$0\in\mathbb{H}^n$.
It goes without saying  that any other Isoperimetric Profile  can be obtained from this one, by left-translations and intrinsic dilations.

Let $\rho:=\|z\|=\sqrt{\sum_{i=1}^n (x_i^2+y_i^2)}$ be the norm of
$z=(x_1,y_1,...,x_i,y_i,...,x_n,y_n)\in\R^{2n}$ and let
$u_0:\overline{B_1(0)}:=\{z\in\R^{2n}: 0\leq\rho\leq
1\}\longrightarrow\R$ be the radial function given  by
\begin{equation}\label{u0}u_0(z)=\frac{\pi}{8} +
\frac{\rho}{4}\sqrt{1-\rho^2}-\frac{\rho}{4}\arcsin\rho =:u_0(\rho).\end{equation}Setting \[\Sph^{\pm}:=\left\{p=\exp\left(z, t\right)\in\mathbb{H}^n: t=\pm
 u_0(z),\,\,\forall\,\,z\in\overline{B_1(0)}\right\},\] we call {\it Heisenberg unit Isoperimetric Profile} $\Sph$
 the compact hypersurface built by gluing together   $\Sph^{+}$
and  $\Sph^{-}$, i.e.
 $\Sph=\Sph^{+}\cup\Sph^{-}$. Since
$\nabla_{\R^{2n}} u_0={u_0}'(\rho)\frac{z}{\rho}$, it follows that
the Euclidean unit normal along $\Sph^\pm$ is given by
${\rm n}^{\pm}\eu=\frac{\left(-\nabla_{\R^{2n}}
 u_0,\pm 1\right)}{\sqrt{1+\|\nabla_{\R^{2n}}
 u_0\|^2}}$. This  implies that
\begin{eqnarray*}\nu^{\pm}=\frac{\left(-\nabla_{\R^{2n}}
 u_0 \pm \frac{z^{\perp}}{2},\pm 1\right)}{\sqrt{1+\|\nabla_{\R^{2n}} u_0\|^2+\frac{\rho^2}{4}}},\qquad|\PH(\nu^{\pm})|=
\frac{\sqrt{\|\nabla_{\R^{2n}} u_0\|^2+\frac{\rho^2}{4}}}
{\sqrt{1+\|\nabla_{\R^{2n}}
u_0\|^2+\frac{\rho^2}{4}}}.\end{eqnarray*}
Using $u_0'(\rho)
=\frac{-\rho^2}{2\sqrt{1-\rho^2}}$, we get that
\begin{eqnarray*}{\nu}^{\pm}\cc=\frac{\left(-\nabla_{\R^{2n}} u_0
\pm \frac{z^{\perp}}{2}\right)}{\sqrt{\|\nabla_{\R^{2n}}
u_0\|^2+\frac{\rho^2}{4}}}=z\pm\frac{\sqrt{1-\rho^2}}{\rho}z^\perp.\end{eqnarray*}
Hence, by definition, it follows that
\begin{eqnarray*}\perh\res\Sph^{\pm}=|\PH(\nu^{\pm})|\,\sigma^{2n}\rr\res\Sph^{\pm}
={\sqrt{\|\nabla_{\R^{2n}} u_0\|^2+\frac{\rho^2}{4}}}\,dz\res
B_1(0)=\frac{\rho}{2\sqrt{1-\rho^2}}\,dz\res
B_1(0).\end{eqnarray*}It is not difficult to compute the horizontal mean curvature of $\Sph$. In fact $\MS=-\div\cc\nn=-2n$. The so-called \textquotedblleft characteristic direction\textquotedblright along $S$ is the horizontal tangent vector field given by
\begin{eqnarray*}({\nu}^{\pm}\cc)^\perp=\left(-z\pm\frac{\sqrt{1-\rho^2}}{\rho}z^\perp\right)^\perp
=z^\perp\mp\frac{\sqrt{1-\rho^2}}{\rho}z.\end{eqnarray*}

\begin{oss}In the 1st
Heisenberg group $\mathbb{H}^1$, the geometric quantity
$\langle\gc_{\nn\op}\nn\op,\nn\rangle$, which coincides with the horizontal mean
curvature $\MS$ of $S$,  turns out to be the
 \rm geodesic curvature \it (see  \cite{Ch1}, p.
203.)  of any  smooth horizontal path
$\gamma:]-\epsilon, \epsilon[\subset\R\longrightarrow S$ such that
$\dot{\gamma}=\nn\op(\gamma)$; see also Remark \ref{njn}.
\end{oss}

The (weighted) vertical component of the Riemannian normal is given by
\begin{eqnarray*}\varpi^{\pm}=\frac{\nu^\pm_{T}}{|\PH\nu|}=\frac{\pm 1}{\|\nabla_{\R^{2n}} u_0+ \frac{z^{\perp}}{2}\|}=\frac{\pm 1}
{\sqrt{\|\nabla_{\R^{2n}} u_0\|^2+\frac{\rho^2}{4}}}=\frac{\pm
1}{\sqrt{{\frac{\rho^4}{4(1-\rho^2)}+ \frac{\rho^2}{4}}}}=\pm
2\frac{\sqrt{1-\rho^2}}{\rho}.\end{eqnarray*} Note that
$\dg\varphi=\nabla_{\R^{2n}}\varphi$ for every function $\varphi:\mathbb H^n\longrightarrow\R$ independent of $t$. Therefore
\begin{eqnarray*}\dg\varpi^\pm=\nabla_{\R^{2n}}\varpi^\pm=
\pm\frac{\partial}{\partial\rho}\left(2\frac{\sqrt{1-\rho^2}}{\rho}\right)\frac{z}{\rho}
=\mp
\left(\frac{2}{\rho^2\sqrt{1-\rho^2}}\right)\frac{z}{\rho}\end{eqnarray*}and
\begin{eqnarray}\label{derombar}\frac{\partial\varpi}{\partial\nn^\perp}=
\mp\frac{2}{\rho^2\sqrt{1-\rho^2}}\left\langle\frac{z}{\rho},({\nu}^{\pm}\cc)^\perp\right\rangle=\frac{1}{\rho\sqrt{1-\rho^2}}\frac{2\sqrt{
1-\rho^2}}{\rho}=\frac{2}{\rho^2}.\end{eqnarray}

\begin{no}Let
${\kappa}:
\Sph\longrightarrow\R$, where $\kappa^\pm:=\kappa|_{\Sph^\pm}=\pm\frac{\sqrt{1-\rho^2}}{\rho}.$ Moreover, we set $g\cc:=\langle z,\nn\rangle$ and $g\cc^\perp:=\langle
z,\nn^{\perp}\rangle$. The function $g\cc$ is called {\rm horizontal
support function} associated with $\Sph$.\end{no}

Throughout the next proofs, we shall choose an adapted o.n.
frame centered at a point $p\in\Sph$  as in Lemma \ref{Sple}. For the sake of
simplicity, we only consider the case of the north hemisphere
$\Sph^+$. In this case, one has $\nn=\nn^+=z+ \kappa z^\perp $ and
$\kappa'_\rho=\frac{\partial\kappa^+}{\partial\rho}=-\frac{1}{\rho^2\sqrt{1-\rho^2}}$. Let us state a key-identity of this paper.

\begin{lemma}\label{ccla}We have $\Delta\ss\kappa=-\frac{2n-4}{\rho^2}\kappa$.\end{lemma}\begin{proof}
Setting $z_i:=\langle z,\tau_i\rangle$, we have
$\tau_i(\kappa)=\kappa'_\rho\frac{
z_i}{\rho}=-\frac{z_i}{\rho^3\sqrt{1-\rho^2}}$. So we
compute\begin{eqnarray*}\Delta\ss\kappa &=&\sum_{i\in
I\ss}\tau_i\left(\tau_i(\kappa)\right)=-\sum_{i\in
I\ss}\left\langle\grad\cc\left(\frac{z_i}{\rho^3\sqrt{1-\rho^2}}\right),\tau_i\right\rangle\\&=&-\sum_{i\in
I\ss}\left(\frac{\tau_i(z_i)}{\rho^3\sqrt{1-\rho^2}}+
z_i\tau_i\left(\frac{1}{\rho^3\sqrt{1-\rho^2}}\right)\right)\\&=&-\sum_{i\in
I\ss}\left(\frac{\delta_{ii}+\langle\gc_{\tau_i}\tau_i,z\rangle}{\rho^3\sqrt{1-\rho^2}}-
\frac{z^2_i}{\rho}\cdot\frac{(3\rho^2(1-\rho^2)-\rho^4)}{\rho^6(1-\rho^2)^{\frac{3}{2}}}\right)\qquad\mbox{($\delta_{ij}$ is the Kroneker delta)}
\\&=&-\left(\frac{2n-1 + g\cc\MS}{\rho^3\sqrt{1-\rho^2}}-
{|z\ss|^2}\cdot\frac{(3\rho^2-4\rho^4)}{\rho^7(1-\rho^2)^{\frac{3}{2}}}\right)\\&=&-\left(\frac{2n-1
-2n\rho^2}{\rho^3\sqrt{1-\rho^2}}-
\frac{(3-4\rho^2)}{\rho^3\sqrt{1-\rho^2}}\right)=-\frac{2n-4}{\rho^2}\kappa,
\end{eqnarray*}where we have used the identity $|z\ss|^2=\rho^2(1-\rho^2)$. This achieves the proof.\end{proof}

\begin{lemma}\label{SeB}Let $B\cc$ be the horizontal 2nd fundamental form of $\Sph$. Then $\|B\cc\|^2\ngr = 4 +
\frac{2n-2}{\rho^2}$ and $\|S\cc\|^2\ngr=Q=2n+2$.
\end{lemma}

\begin{proof}
By applying (i) of Lemma \ref{Doca} we have $\|B\cc\|^2\ngr=-\sum_{i\in
I\ss}\langle\gc_{\tau_i}\gc_{\tau_i}\nn,\nn\rangle$ and
since\begin{eqnarray*}\gc_{\tau_i}\nn=\gc_{\tau_i}\left(z+\kappa
z^\perp\right)=\tau_i + \tau_i(\kappa)z^\perp -\kappa
C^{2n+1}\cc\tau_i,\end{eqnarray*}we get that
\begin{eqnarray}\nonumber\|B\cc\|^2\ngr&=&-\sum_{i\in
I\ss}\langle\gc_{\tau_i}\gc_{\tau_i}\nn,\nn\rangle\\\nonumber&=&-\sum_{i\in
I\ss}\left\langle\gc_{\tau_i}\left(\tau_i + \tau_i(\kappa)z^\perp
-\kappa
C^{2n+1}\cc\tau_i\right),\nn\right\rangle\\\nonumber&=&\sum_{i\in
I\ss}\left(-\langle\gc_{\tau_i}\tau_i,\nn\rangle-\Delta\ss\kappa\langle
z^\perp,\nn\rangle+2 \tau_i(\kappa)\langle
C^{2n+1}\cc\tau_i,\nn\rangle+\kappa\langle
C^{2n+1}\cc\gc_{\tau_i}\tau_i,\nn\rangle\right)\\\label{oiiooiio}&=&\left(-\MS+g\cc^\perp\Delta\ss\kappa+2
\langle C^{2n+1}\cc\qq\kappa,\nn\rangle\right),\end{eqnarray}where we have used the identity $\langle
C^{2n+1}\cc\gc_{\tau_i}\tau_i,\nn\rangle=0$, which holds for every $i\in I\ss$.
Since $C^{2n+1}\cc\nn\in\HS$, the last identity can  be proved by using
an adapted horizontal frame, as in Lemma \ref{Sple}. Moreover
\begin{eqnarray}\label{pooppoop}\langle
C^{2n+1}\cc\qq\kappa,\nn\rangle =
\frac{\kappa'_\rho}{\rho}\langle C^{2n+1}\cc
z\ss,\nn\rangle =\frac{\kappa'_\rho}{\rho}\langle
C^{2n+1}\cc \left(z-g\cc\nn\right),\nn\rangle=-\rho\kappa\kappa'_\rho.\end{eqnarray}At this point, using Lemma \ref{ccla} together with \eqref{oiiooiio},
\eqref{pooppoop} and the identity
$g\cc^\perp=-\rho\sqrt{1-\rho^2}$, yields
\begin{eqnarray*}\|B\cc\|^2\ngr=\left (2n
+\rho\sqrt{1-\rho^2}\cdot
\frac{2n-4}{\rho^2}\cdot\kappa+2\kappa\frac{1}{\rho\sqrt{1-\rho^2}}\right)=4
+\frac{2n-2}{\rho^2}\end{eqnarray*}which proves the first claim. Finally, since $\|B\cc\|^2\ngr=\|S\cc\|^2\ngr +
\|A\cc\|^2\ngr$ and $\|A\cc\|^2\ngr=\frac{2n-2}{4}\varpi^2$, using  $\varpi^2=4\frac{{1-\rho^2}}{\rho^2}$ yields
$\|S\cc\|^2\ngr=Q=2n+2$.
\end{proof}

Note that we have found the Gram-norm of $S\cc$  in an indirect way. However, we can be more precise.
To this aim, we first compute \begin{eqnarray*}\mathcal{J}\cc\nn^\pm=\mathcal{J}\cc(z\pm \kappa z^\perp)={\bf Id}\pm\dfrac{\kappa'_\rho}{\rho}z^\perp\otimes z\mp \kappa C\cc^{2n+1}.\end{eqnarray*}
By Lemma \ref{fgh2}, we have
$S\cc(X, Y)=- \langle W\cc X,Y\rangle-\dfrac{\varpi}{2}\langle C\ss^{2n+1}X,Y\rangle$ for every $X, Y\in\XX^1\ss$.
Furthermore, let $\mathcal{F}\ss=\{\tau_i: i\in I\ss\}$ be any horizontal tangent o.n. frame. Then  \begin{eqnarray*}-S\cc(\tau_i,\tau_j)&=& W\cc(\tau_i, \tau_j)+\dfrac{\varpi}{2}\langle C\ss^{2n+1}\tau_i,\tau_j\rangle\\&=&\langle \mathcal{J}\cc\nn^\pm\tau_i, \tau_j\rangle+\dfrac{\varpi}{2}\langle C\ss^{2n+1}\tau_i,\tau_j\rangle\\&=&\langle \mathcal{J}\cc\nn^\pm\tau_i, \tau_j\rangle\pm {\kappa}\langle C\ss^{2n+1}\tau_i,\tau_j\rangle\\&=&\delta_{ij}\pm\dfrac{\kappa'_\rho}{\rho}\left\langle \left(z^\perp\otimes z\right)\tau_i, \tau_j\right\rangle \\&=&\delta_{ij}\mp\dfrac{1}{\rho^3\sqrt{1-\rho^2}}\langle z, \tau_i\rangle\langle z^\perp, \tau_j\rangle.\end{eqnarray*}
With no loss of generality,  take $\tau_2:=\nn^\perp$. By a simple computation, we get $-S\cc(\tau_2,\tau_2)=2$.
Furthermore, note that for any $X\in\HS$ such that $\langle X,\nn^\perp\rangle=0$, one has
must have  $\langle X, z\rangle=\langle X, z^\perp\rangle=0$.
Hence, for any  o.n. frame $\mathcal{F}\ss=\{ \tau_2,...,\tau_{2n}\}$ for $\HS$ such that $\tau_2=\nn^\perp $,  $S\cc$ turns out to be the diagonal matrix of order $(2n-1)$ given by
$S\cc:={\rm Diag}[-2,-1,-1,...,-1].$ Thus we have the following:

\begin{Prop}The principal horizontal curvatures of the Heisenberg unit Isoperimetric profile $\Sph$ are the numbers $\kappa_2=-2$, $\kappa_3=...=\kappa_{2n}=-1$. Furthermore, we have that any horizontal tangent o.n. frame $\mathcal{F}\ss=\{\tau_i: i\in I\ss\}$ such that $\tau_2=\nn^\perp$ turns out to be a system of eigenvectors of $S\cc$.

 \end{Prop}

 The principal horizontal curvatures reflect the geometric construction of the  unit Isoperimetric profile $\Sph$. Indeed, any Isoperimetric profile is generated by rotating a CC-geodesic $\gamma$ joining two consecutive points belonging
to a vertical $T$-line; these points are  the South and North poles.
 Note that the number $\kappa_2$ just express a \textquotedblleft curvature parameter \textquotedblright
which uniquely determines all CC-geodesics joining the South pole to the North pole; see also Remark \ref{njn}.  As already said, this parameter
is a special feature of CC-geodesics. Note also that the other principal horizontal curvatures express the rotational symmetry of $\Sph$ with respect to the $T$ axis. We end this section with a useful remark about the operator $\lh$.

\begin{oss}\label{RemFinT}
Let $\varphi:\Sph\longrightarrow\R$ be a smooth function and  consider its restrictions  to the hemispheres
$\Sph^\pm=\left\{p\equiv(z, t)\in\mathbb{H}^n: t=\pm u_0(z),\,z\in
\overline{B_1(0)}\right\}$. Since $\varphi(x)=\varphi(\exp(z, \pm
u_0(z)))$, we may thinking of $\varphi^\pm:=\varphi|_{\Sph^\pm}$ as  functions
of the variable $z\in \overline{B_1(0)}$. Now let
$\varphi:\overline{B_1(0)}\setminus\{0\}\longrightarrow\R$ be a smooth function and fix spherical coordinates on
$\overline{B_1(0)}\setminus\{0\}$, i.e.
 $(\rho, \xi)\in]0,1]\times \mathbb{S}^{2n-1} $. With a slight abuse of notation,
 every  function $\varphi:\Sph^\pm\setminus\{\mathcal{N, S}\}\longrightarrow\R$ will be regarded as  a function of
 the variables $(\rho, \xi)\in]0, 1]\times\mathbb{S}^{2n-1}$. Setting $
 \zeta:=\frac{z^{\perp}}{\rho}\in \mathbb{S}^{2n-1}$, the operator $\lh$ on $\Sph$ takes the
following
form:\small\begin{eqnarray}\label{ffs}\lh\varphi=\underbrace{(1-\rho^2)\varphi''_{\rho\rho}+\frac{2n
-(2n+1)\rho^2}{\rho}\varphi'_\rho}_{\rm Radial\, Operator}-
\underbrace{2\rho\sqrt{1-\rho^2}\varphi''_{\zeta \rho}}_{\rm Mixed
\,Derivatives}+\underbrace{\frac{1}{\rho^2}\Delta_{\mathbb{S}^{2n-1}}\varphi
-(1-\rho^2)\varphi''_{\zeta\zeta}-(Q-1)
\sqrt{1-\rho^2}\varphi'_\zeta}_{\rm
Angular\,Operator},\end{eqnarray}\normalsize where $\Delta_{\mathbb{S}^{2n-1}}$ denotes the Laplace operator on the Sphere
$\mathbb{S}^{2n-1}$; see \cite{Monted}. Note also
that
\begin{equation}\label{rftr}
\div_{\TT\mathbb{S}^{2n-1}}\zeta=0;\end{equation} see Lemma 2.15
in \cite{Monted}.
\end{oss}

\section{1st and 2nd variation of $\perh$ along compact hypersurfaces}
\label{1e2varfor}

Let $S\subset\mathbb{H}^n$ be a $\cont^2$-smooth compact closed
hypersurface oriented by its unit normal vector $\nu$ and let
$\UU\subset S\setminus C_S$ be a non-characteristic open set in
the relative topology. We assume that the boundary $\partial\UU$
is a $(2n-1)$-dimensional (piecewise) $\cont^1$-smooth submanifold
oriented by its outward unit normal vector $\eta$. We say that a
smooth map $\vartheta:]-\epsilon, \epsilon[\times
\UU\longrightarrow\mathbb{H}^n$ is a {\it variation} of $\UU$ if
the following hold:
{\rm(i)} every
$\vartheta_t:=\vartheta(t,\cdot):\UU\rightarrow\GG$ is an
immersion;\,\,{\rm(ii)} $\vartheta_0={\rm Id}_\UU$. By definition, the variation vector of $\vartheta$ is given by $W:=\frac{\partial
\vartheta}{\partial
t}\big|_{t=0}=\widetilde{W}|_{t=0},$ where
$\widetilde{W}=\vartheta_{\ast}\frac{\partial}{\partial t}$.
Let $\perht$ denote the $\HH$-perimeter measure along
$\vartheta_t(\UU)$ and set $\Gamma(t):=\vartheta_t^\ast\perht\in\bigwedge^{2n}
(\TT^\ast\UU),$ for every $t\in]-\epsilon,\epsilon[$. Note that
$\Gamma(t)$ is a 1-parameter family  of $2n$-forms along $\UU$. The 1st and 2nd variation formulas of the
$\HH$-perimeter $\perh$ under the variation $\vartheta$ are  given
by $I_\UU(W,\per)=
\int_{\UU}\dot{\Gamma}(0)$ and $II_\UU(W,\per)=
\int_{\UU}\ddot{\Gamma}(0)$.

In \cite{Monteb} we proved in a more general context, the
following:
\begin{teo}[see \cite{Monteb}]\label{12VF}Let $S\subset\mathbb{H}^n$ be a $\cont^2$-smooth
hypersurface oriented by its unit normal vector $\nu$ and let
$\UU\subset S\setminus C_S$ be a non-characteristic relatively
compact open set having piecewise $\cont^1$-smooth boundary
$\partial\UU$ oriented by its outward unit normal vector $\eta$.
Let $\vartheta_t$ be a variation of $\UU$ with variation vector
$W$ and set
 $\cn:=\frac{\langle W ,\nu
 \rangle}{|\PH\nu|}$. Then
\[I_\UU(W,\perh) = - \int_{\UU}\MS
\cn\,\perh +  \int_{\partial\UU}
  \langle W, \eta \rangle\, |\PH{\nu}|\,\sigma^{2n-1}\rr\]
Moreover, if $\UU$ has constant horizontal mean curvature $\MS$,
then \begin{eqnarray*}
II_\UU(W,\per)=\int_{\UU}\left(-\MS\,\widetilde{W}(\cn_t)\big|_{t=0}+
|\qq\cn|^2+\cn^2\left( (\MS)^2-\|S\cc\|^2\ngr+
2\frac{\partial\varpi}{\partial\nn^\perp} -\frac{n+1}{2}\varpi^2
\right)\right)\perh\\+
\int_{\partial\UU}\left(\left\langle\left(-\cn\,\qq\cn +
[\WW^{\nu^t},\WW^{\TT}]^{\TT}\big|_{t=0}\right),\eta\right\rangle\,|\PH\nu|+
 \left(\div\ts(|\PH\nu| {W}^{\!\TT})- \MS \langle{W},\nu\rangle
\right)\langle{W}^{\!\TT},\eta\rangle\right)\sigma^{2n-1}\rr,\end{eqnarray*}
where
$W^\TT$ denotes the tangential component of $W$ along
 $\UU$ and $\WW^{\TT},\,\WW^{\nu^t}$ denote, respectively,
tangential and normal components of $\WW$ along $\UU_t$.
Moreover, we have set $\cn_t:=\frac{\langle \widetilde{W}
,\nu^t
 \rangle}{|\widetilde{\PH}\nu^t|}$, where  $\nu^t$ is the Riemannian unit normal vector along
$\vartheta_t(\UU)$ and $\widetilde{\P\cc}$ is the orthogonal
projection onto $\HH$ at $\vartheta_t(p)$, for every $p\in
S$.\end{teo}

In \cite{Monteb} we stated all results for $\cin$-smooth hypersurfaces. If instead we assume that
$S$ is only in $\cont^2$, some of the computations in the proof  of the previous theorem should be understood in the sense of the distribution theory. More precisely, the proof uses explicitly the so-called {\rm curvature $2$-forms}  associated with an adapted o.n. frame along $S$; see \cite{Ch1}, \cite{Spiv}, or \cite{Monte, Monteb}. However, one can also assume that $S$ is in  $\cont^3$ and then use an approximation argument
in order to extend the final formula, where no third derivatives occur, to the $\cont^2$ case.

\begin{oss}\label{carnei}Let $S\subset \mathbb{H}^n$ be a $\cont^2$-smooth compact closed
hypersurface. In this case it turns out that $\dim_{\rm
Eu-Hau}({\rm Car}_S)\leq n$; see \cite{Bal3}. Just for the case
$n=1$, we shall further suppose that {\rm $C_S$ is contained in a
finite union of $\cont^1$-smooth horizontal curves}. As already
said, under these assumptions one can show that there exists a
family $\{\UU_\epsilon\}_{\epsilon\geq 0}$ of open subsets of $S$,
with piecewise $\cont^1$-smooth boundaries, such
that: {\rm(i)} $C_S\Subset\UU_\epsilon$ for
every $\epsilon>0$; \,\, {\rm(ii)}
$\sigma^{2n}\rr(\UU_\epsilon)\longrightarrow 0$ for
$\epsilon\rightarrow
0^+$; \,\,{\rm(iii)} $\int_{\partial\UU_\epsilon}|\PH\nu|\,\sigma^{2n-1}\rr\longrightarrow
0$ for $\epsilon\rightarrow 0^+$.
\end{oss}

Set $S_\epsilon:=S\setminus \UU_\epsilon$. Later on
we shall discuss the validity of Theorem \ref{12VF} for
$\cont^2$-smooth closed compact hypersurfaces $S$ having non-empty
characteristic set. To this end, let us consider the following
limits (if they
exist):\begin{eqnarray*}I_S(W,\per):=\lim_{\epsilon\rightarrow
0^+}I_{S_\epsilon}(W,\per),
\qquad II_S(W,\per):=\lim_{\epsilon\rightarrow 0^+}II_{S_\epsilon}(W,\per).\end{eqnarray*}
Note that $I_{S_\epsilon}(W,\per)$ and
$II_{S_\epsilon}(W,\per)$ represent the 1st and 2nd variation of
$\perh$ along the 1-parameter family $\{S_\epsilon\}_{\epsilon>0}$
of non-characteristic hypersurfaces (with boundary). Since $C_S$ is
a null set with respect to the $\sigma^{2n}\rr$-measure (see
Remark \ref{CSET}) it is clear that, if they exist, the limits
$I_S(W,\per)$ and $I_S(W,\per)$ express the 1st and 2nd variation of
$\perh$ along $S$. For every
$\epsilon>0$
 one has
\begin{eqnarray*}I_S(W, \perh)=I_{S_\epsilon}(W, \perh) + I_{\UU_\epsilon}(W, \perh),\qquad
II_S(W, \perh)=II_{S_\epsilon}(W, \perh) +
II_{\UU_\epsilon}(W, \perh).\end{eqnarray*}

\begin{oss}\label{REMREM} Although the quantities appearing in these
formulas are not well defined at  $C_S$, we could
alternatively compute $I_{\UU_\epsilon}(W, \perh)$ and
$II_{\UU_\epsilon}(W, \perh)$, by using the representation formula
$\perh=|\PH\nu|\,\sigma^{2n}\rr$, and then show
that:\begin{eqnarray}\label{lim12v}\lim_{\epsilon\rightarrow
0^+}I_{\UU_\epsilon}(W,\per)=0, \qquad \lim_{\epsilon\rightarrow
0^+}II_{\UU_\epsilon}(W,\per)=0.\end{eqnarray}Since
$I_{\UU_\epsilon}(W,\perh):= \int_{\UU_\epsilon}\dot{\Gamma}(0)$,
$ II_{\UU_\epsilon}(W,\perh):=
\int_{\UU_\epsilon}\ddot{\Gamma}(0)$, where
$\Gamma(t):=\vartheta_t^\ast\perht$,  we need to compute
\begin{eqnarray*}\dot{\Gamma}(0)&=&
\left(\frac{d}{dt}|\widetilde{\P\cc}\nu^t|\right)\bigg|_{t=0}\,\sigma^{2n}\rr
+ |\P\cc\nu|
\left(\frac{d}{dt}(\sigma^{2n}\rr)_t\right)\bigg|_{t=0},\\\ddot{\Gamma}(0)&=&
\left(\frac{d^2}{dt^2}|\widetilde{\P\cc}\nu^t|\right)\bigg|_{t=0}\,\sigma^{2n}\rr
+
2\left(\frac{d}{dt}|\widetilde{\P\cc}\nu^t|\right)\bigg|_{t=0}\left(\frac{d}{dt}(\sigma^{2n}\rr)_t\right)\bigg|_{t=0}
+|\P\cc\nu|\left(\frac{d^2}{dt^2}(\sigma^{2n}\rr)_t\right)\bigg|_{t=0}.\end{eqnarray*}
Note that $\frac{d}{dt}(\sigma^{2n}\rr)_t$ and
$\frac{d^2}{dt^2}(\sigma^{2n}\rr)_t$,  evaluated at ${t=0}$, express
the ``infinitesimal'' 1st and 2nd variation of the Riemannian Area
$\sigma^{2n}\rr$; see \cite{Spiv} or \cite{Ch1}.
Nevertheless, this analysis goes beyond the scopes of this paper
and below we shall
 discuss a different approach and other
  results valid in Heisenberg groups.\end{oss}
By applying the 1st variation of $\perh$ to
$\UU=S_\epsilon\,(\epsilon>0)$ and (iii) of Remark \ref{carnei},
we see that
 the boundary integral tends to $0$ as long as
$\epsilon\rightarrow 0^+$, i.e.
$$\lim_{\epsilon\rightarrow 0^+}\int_{\partial S_\epsilon}
  \langle W, \eta \rangle\,
  |\PH{\nu}|\,\sigma^{2n-1}\rr=0.$$
Therefore, it remains to study the convergence of the integral
along the interior of
  $S_\epsilon$ as long as $\epsilon\rightarrow 0^+$. By definition
  $\MS=-\div\ss\nn$ and it turns out that
    $\div\ss\nn=\div\cc\nn$. So we have\begin{equation}\label{gghh}
    -\MS=\div\cc\nn=\div\cc\left(\frac{\PH\nu}{|\PH\nu|}\right)
  =\frac{\div\cc(\PH\nu)-\langle\grad\cc|\PH\nu|, \nn\rangle}{|\PH\nu|}.\end{equation}
 By noting that\footnote{Since $S$ is $\cont^2$-smooth, $\nu$ is of
class $\cont^1$ everywhere on $S$.} $|\PH\nu|$ is Lipschitz
continuous at $C_S$, it follows that $\MS\in L^1(S, \perh)$.  More precisely, for every $\epsilon>0$ one
has\begin{eqnarray*}\int_{\UU_\epsilon}|\MS| \,\perh  =
\int_{\UU_\epsilon}\big| \div\cc(\PH\nu)-\langle\grad\cc|\PH\nu|,
\nn\rangle \big|\,\sigma^{2n}\rr
\leq C\,\sigma^{2n}\rr(\UU_\epsilon),
\end{eqnarray*}where $C$ is a constant only dependent on (the
Lipschitz constant of) $\PH\nu$. So
$I_{\UU_\epsilon}(W,\per)\longrightarrow 0$ as long as
${\epsilon\rightarrow 0^+}$ and we finally get that
\begin{eqnarray}\label{fv}I_S(W,\per)=\lim_{\epsilon\rightarrow
0^+}I_{S_\epsilon}(W,\per)=- \int_{S}\MS \cn\,\perh.\end{eqnarray}

 We also need the following fact:
\begin{lemma}Let $\varphi:S\longrightarrow\R$ be a
piecewise smooth function such that
$\int_S\varphi\,\sigma^{2n}\rr=0.$ Then there exists a
volume-preserving normal variation\footnote{This means that
$\widetilde{W}=\vartheta_{\ast}\frac{\partial}{\partial
 t}$ is parallel to $\nu^t$ for every $t\in]-\epsilon, \epsilon[$.} $\vartheta_t$ whose
variation vector is $W=\varphi\,\nu$. If $\varphi=0$ along
$\partial S$ we can always assume that the variation fixes the
boundary.
\end{lemma}\begin{proof}This fact is well-known
in the Euclidean setting and its proof applies as well to our
case; see \cite{DC}, Lemma 2.4.\end{proof}

 \begin{oss}\label{rvvf}As in Riemannian Geometry, the
 1st variation of $\perh$ along a compact closed hypersurface $S$ only depends on the normal component of
 the variation vector. For this reason,
  in the sequel only {\it normal
 variations} of $S$ will be considered.
 Without loss of generality, we shall also assume that $W:=\varphi\,|\PH\nu|\,
 \nu$ for some smooth function
 $\varphi:S\longrightarrow\R$. Since $\cn=\frac{\langle W ,\nu
 \rangle}{|\PH\nu|}$, we have $\cn=\varphi$. \end{oss}

Let $S$ be of constant horizontal mean curvature
$\MS$ (at each point of $S\setminus C_S$). Under the previous
assumptions, the 2nd variation of $\perh$ along $S_\epsilon$ is
given by
\begin{eqnarray*}
  II_{S_\epsilon}(W,\per)&=&\int_{S_\epsilon}\left(-\MS\widetilde{W}
  (\varphi_t)\big|_{t=0}+ |\qq\varphi|^2+\varphi^2\left(\MS^2-\|S\cc\|^2\ngr+
2\frac{\partial\varpi}{\partial\nn^\perp} -\frac{n+1}{2}\varpi^2
\right) \right)\perh\\&&- \int_{\partial S_\epsilon}
\varphi\langle\qq\varphi, \eta\rangle
|\PH\nu|\,\sigma^{2n-1}\rr.\end{eqnarray*}

\begin{lemma}\label{sv}Let $S\subset\mathbb{H}^n$ be a $\cont^2$-smooth
compact closed hypersurface  with constant horizontal mean
curvature $\MS$. If $\frac{1}{|\PH\nu|}\in
L^1(S, {\sigma^{2n}\rr})$, then
 $II_S(W, \perh)=\lim_{\epsilon\rightarrow
0^+}II_{S_\epsilon}(W,\per)$, where the variation vector $W$ is
chosen as in Remark \ref{rvvf}.
\end{lemma}
Notice that  $\frac{1}{|\PH\nu|}\in
L^1(S, {\sigma^{2n}\rr}) \Longleftrightarrow\frac{1}{|\PH\nu|}\in
L^2(S, {\perh})$.
\begin{proof}We claim that the
boundary integral tends to $0$ as long as $\epsilon\rightarrow
0^+$, i.e.
$$\lim_{\epsilon\rightarrow 0^+}\int_{\partial S_\epsilon}\varphi\langle\qq\varphi, \eta\rangle
|\PH\nu|\,\sigma^{2n-1}\rr=0.$$Since $\partial
S_\epsilon=\partial\UU_\epsilon$, using (iii) of Remark
\ref{carnei} we get $\int_{\partial
S_\epsilon}|\PH\nu|\,\sigma^{2n-1}\rr\longrightarrow 0$ for
$\epsilon\rightarrow 0^+$ and the claim follows since
$|\varphi\langle\qq\varphi,
\eta\rangle|\leq\frac{1}{2}\|\qq\varphi^2\|_{L^\infty(S)}$. Let us study the integral along the interior of
  $S_\epsilon$.
\begin{itemize}\item Since \[-\MS\widetilde{W}(\varphi_t)\big|_{t=0}=-\MS \varphi\,
|\PH\nu|\left(\frac{d\varphi_t}{dt}\bigg|_{t=0}\right),\]and since
$\MS$ is constant along $S\setminus C_S$, it follows that the 1st addend can be
integrated over all of $S$; since $|\qq\varphi|\leq
\|\qq\varphi\|_{L^\infty(S)}$, the same holds true for
the 2nd addend.

\item One has
 $\|S\cc\|^2\ngr=\sum_{i,
j\in I\ss}\left(\dfrac{\langle\gc_{\tau_i}\nn, \tau_j\rangle +\langle\gc_{\tau_j}\nn, \tau_i\rangle }{2}\right)^2$, where $\{\tau_i: i\in I\ss\}$ is an o.n. basis of
$\HS$; see Definition \ref{movadafr}. Note also that
$\langle\gc_{\tau_i}\nn,
\tau_j\rangle^2=\frac{\langle\gc_{\tau_i}(\PH\nu),
\tau_j\rangle^2}{|\PH\nu|^2}$ for every $i, j\in I\ss$. Since $\frac{1}{|\PH\nu|^2}\in
L^1(S, {\perh})$, it follows that $\|S\cc\|^2\ngr\in L^1(S,
\perh)$. \item The 5th addend can be integrated
 over all of $S$, since the following estimate
\begin{eqnarray*}\left|\frac{\partial\varpi}{\partial\nn^\perp}\right|&=&
\left|\frac{\nn^\perp({\nu}_{_{\!T}})|\PH\nu|-\nn^\perp(|\PH\nu|){\nu}_{_{\!T}}}{|\PH\nu|^2}\right|
\lesssim \frac{1}{|\PH\nu|^2}\end{eqnarray*} holds true near $C_S$. Finally, the 6th term satisfies $\varpi^2 \leq
\frac{1}{|\PH\nu|^2}$.

\end{itemize}Using the previous remarks and the smoothness of $\varphi$ (over all of $S$), the thesis easily follows.
\end{proof}

\begin{corollario}Let $S\subset\mathbb{H}^n$ be a $\cont^2$-smooth
compact closed hypersurface.  Let $\vartheta_t$ be a normal
variation, having variation vector $W=\varphi\,|\PH\nu|\,
 \nu$, for some smooth function
 $\varphi:S\longrightarrow\R$. Then \begin{eqnarray}\label{Finale1v}I_S(W,\per)=- \int_{S}\MS
\varphi\,\perh.\end{eqnarray}If
 $\frac{1}{|\PH\nu|}\in L^2(S, {\perh})$ and $S$ has constant horizontal
mean curvature, then
\begin{eqnarray}\label{Finale2v}
II_S(W,\per)
=\int_{S}\left(-\MS\,\widetilde{W}(\varphi_t)\big|_{t=0}+
|\qq\varphi|^2 +\varphi^2\left(\MS^2-\|S\cc\|^2\ngr+
2\frac{\partial\varpi}{\partial\nn^\perp} -\frac{n+1}{2}\varpi^2
\right)\right)\perh.\,\,\end{eqnarray}
\end{corollario}

Finally, we discuss another point of view, which can be used in order to extend Theorem \ref{12VF} even in more general situations. In particular, we would like to use (normal) variations which can be (possibly) singular at $C_S$.
For the sake of simplicity, \it we only consider the case $n>1$. \rm In view of Remark \ref{CSET}, this implies that $\dim   C_S<2n-2$.
As already said, the validity of the 1st and 2nd variation formulas for $\perh$ up to the characteristic set $C_S$, can be formulated in terms of a limit procedure.
More precisely, let $S$ be a compact hypersurface of class $\cont^2$ without boundary and
 let us set $S_\epsilon=S\setminus \UU_\epsilon$, where $\{\UU_\epsilon\}_{\epsilon\geq 0}$ is a family  of open subsets of $S$,
with (piecewise) $\cont^1$-smooth boundaries, such
that: {\rm(i)} $C_S\Subset\UU_\epsilon$ for
every $\epsilon>0$; \,\, {\rm(ii)}
$\sigma^{2n}\rr(\UU_\epsilon)\longrightarrow 0$ for
$\epsilon\rightarrow
0^+$; \,\, {\rm(iii)}
$\sigma^{2n-1}\rr(\partial\UU_\epsilon)\longrightarrow 0$ for
$\epsilon\rightarrow
0^+$. Note that these  sets shrink around $C_S$, as long as $\epsilon\rightarrow 0$. Furthermore, let $\vartheta_t$  be a normal variation of $S$ with variation vector $W=\varphi |\PH \nu| \nu$, for some function $\varphi:S\longrightarrow\R$. Nevertheless, \it we do not assume that $\varphi$ is smooth over all $S$, but only on $S \setminus C_S$. \rm Under these assumptions,
for every $\epsilon>0$ the 1st and 2nd variation formulas on the non-characteristic hypersurfaces $S_\epsilon$ are given  by
\begin{eqnarray*}I_{S_\epsilon}(W,\perh) &=& - \int_{S_\epsilon}\MS
\varphi \,\perh, \\
II_{S_\epsilon}(W,\per)&=& \int_{S_\epsilon}\left(-\MS\,\widetilde{W}(\varphi_t)\big|_{t=0}+
|\qq\varphi |^2+ \varphi^2\left( (\MS)^2-\|S\cc\|^2\ngr+
2\frac{\partial\varpi}{\partial\nn^\perp} -\frac{n+1}{2}\varpi^2
\right)\right)\perh\\&&-
\int_{\partial{S_\epsilon}} \left\langle \varphi\,\qq \varphi ,\eta\right\rangle\,|\PH\nu|\,\sigma^{2n-1}\rr.\end{eqnarray*}
This follows from Theorem \ref{12VF}, since $\vartheta_t$ is a normal variation. Now  consider  the  limits (if they exist):
\begin{eqnarray*}I_S(W,\per):=\lim_{\epsilon\rightarrow
0^+}I_{S_\epsilon}(W,\per),
\qquad II_S(W,\per):=\lim_{\epsilon\rightarrow 0^+}II_{S_\epsilon}(W,\per).\end{eqnarray*}
What is a sufficient condition for the existence of these limits? The previous analysis showed that if $\varphi$ is smooth on all of $S$, then the limits exist (For what concerns the 2nd variation, we also have to assume $\frac{1}{|\PH\nu|}\in L^2(S, {\perh})$.). However, it is enough to require that all integrands are continuous over all of $S$.

\begin{no}\label{notchi}
Let us set \begin{eqnarray*}
\chi_1&:=&-\MS
\varphi \,\perh,\\
 \chi_2&:=&\left(-\MS\,\widetilde{W}(\varphi_t)\big|_{t=0}+
|\qq\varphi |^2+ \varphi^2\left( (\MS)^2-\|S\cc\|^2\ngr+
2\frac{\partial\varpi}{\partial\nn^\perp} -\frac{n+1}{2}\varpi^2
\right)\right)\perh,\\ \chi_3 &:=&-\left\langle \varphi\,\qq \varphi ,\eta\right\rangle\,|\PH\nu|\,\sigma^{2n-1}\rr.
\end{eqnarray*}
\end{no}

\begin{Prop}\label{alterver2} Let $n>1$. Let $S\subset\mathbb H^n$ be a compact hypersurface of class $\cont^2$ without boundary and let $\vartheta_t$ be a normal variation of $S$. Let $W$ be the vector variation of  $\vartheta_t$ and assume that $W=\varphi |\PH \nu| \nu$ for some function $\varphi:S\longrightarrow\R$ which is $\cont^2$-smooth on $S\setminus C_S$. Furthermore, assume that
the differential forms $\chi_1, \chi_2, \chi_3$ are continuous on all of $S$.
Then \begin{eqnarray*} I_S(W,\per)&=&- \int_{S}\MS
\varphi\,\perh,\\II_S(W,\per)
&=&\int_{S}\left(-\MS\,\widetilde{W}(\varphi_t)\big|_{t=0}+
|\qq\varphi|^2 +\varphi^2\left(\MS^2-\|S\cc\|^2\ngr+
2\frac{\partial\varpi}{\partial\nn^\perp} -\frac{n+1}{2}\varpi^2
\right)\right)\perh.\end{eqnarray*}
\end{Prop}
\begin{proof}It is enough to note that, if the differential forms $\chi_1, \chi_2, \chi_3$ are continuous on all of $S$, then the integrals $I_{S_\epsilon}(W,\perh),\,II_{S_\epsilon}(W,\perh),$ turn out to be well-defined (and finite) for every $\epsilon\geq 0$. The thesis follows since the boundaries $\partial S_\epsilon$ converge to the lower dimensional set $C_S$, as long as $\epsilon\rightarrow 0$.

\end{proof}

\subsection{1st and 2nd variation of volume, isoperimetric functional and the notion of
stability}\label{volvarIS}

Let $D\subset\mathbb{H}^n$ be a relatively compact domain with
$C^2$-smooth boundary  $S:=\partial D$. Let
$\imath_D:D\longrightarrow\mathbb{H}^n$ be the inclusion of $D$ in
$\mathbb{H}^n$ and let
 $\vartheta: ]-\epsilon,\epsilon[\times D
\rightarrow \mathbb{H}^n$ be a smooth map. We say that
$\vartheta$ is a {\it variation} of $\imath_D$ if the following
hold: {\rm(i)} every
$\vartheta_t:=\vartheta(t,\cdot):D\rightarrow\mathbb{H}^n$ is an
immersion;\,\, {\rm(ii)} $\vartheta_0=\imath_D$.

Let $\vartheta_t$ be a variation of $D$ with
variation vector $W=\vartheta_{\ast}\frac{\partial}{\partial
t}\big|_{t=0}$ and set
$\WW:=\vartheta_{\ast}\frac{\partial}{\partial t}$. The 1st and
 2nd variation formulas of  $\sigma^{2n+1}\rr\res D$, denoted as $I_D(W,\sigma^{2n+1}\rr)$ and
$II_D(W,\sigma^{2n+1}\rr)$, are given by
$$I_D(W,\sigma^{2n+1}\rr):
=\frac{d}{dt}\left(\int_{D}\vartheta_t^\ast\left(\sigma^{2n+1}\rr\right)\right)\bigg|_{t=0},\qquad
II_D(W,\sigma^{2n+1}\rr):=
\frac{d^2}{dt^2}\left(\int_{D}\vartheta_t^\ast\left(\sigma^{2n+1}\rr\right)\right)\bigg|_{t=0}.$$
Setting
$\left(\sigma^{2n+1}\rr\right)_t:=\vartheta_t^\ast\left(\sigma^{2n+1}\rr\right)$,
we see that
$$\frac{d}{dt}\left(\sigma^{2n+1}\rr\right)_t
=\mathcal{L}_{\WW}\left(\sigma^{2n+1}\rr\right)_t,\qquad\frac{d^2}{dt^2}\left(\sigma^{2n+1}\rr\right)_t=
\mathcal{L}_{\WW}\left(\mathcal{L}_{\WW}\left(\sigma^{2n+1}\rr\right)_t\right).$$By
Cartan's formula for the Lie derivative we compute
\begin{eqnarray*}\mathcal{L}_{\WW}\left(\sigma^{2n+1}\rr\right)_t=d(\WW\LL\left(\sigma^{2n+1}\rr\right)_t)+
\WW\LL\underbrace{d\left(\sigma^{2n+1}\rr\right)_t}_{=0}=\div \WW
\left(\sigma^{2n+1}\rr\right)_t\end{eqnarray*}and by applying Stokes' Theorem
we get that
\begin{eqnarray}\label{1varvol}I_D(W,\sigma^{2n+1}\rr)=\int_D\div
W\,\sigma^{2n+1}\rr=\int_{S}\cn\,\per,\end{eqnarray} where
$\cn=\frac{\langle W ,\nu
 \rangle}{|\PH\nu|}$. For what concerns the 2nd variation of $\sigma^{2n+1}\rr\res
 D$, let us compute
\begin{eqnarray*}
\mathcal{L}_{\WW}\left(\mathcal{L}_{\WW}\left(\sigma^{2n+1}\rr\right)_t\right)&=&
\underbrace{\mathcal{L}_{\WW}\left(d\left(\WW\LL\left(\sigma^{2n+1}\rr\right)_t\right)\right)
=d\left(\mathcal{L}_{\WW}\left(\WW\LL\left(\sigma^{2n+1}\rr\right)_t\right)\right)}_{\mbox{\tiny{since}}\,\,\,\mathcal{L}\circ
d= d\circ\mathcal{L}}\\&=&d\left(\mathcal{L}_{\WW}\left(\cn_t\,
\perth\right)\right)=d\left(\left(\WW(\cn_t)
-(\MS)_t\cn_t^2\right)\perth\right),\end{eqnarray*}where
$\cn_t:=\frac{\langle \WW ,\nu^t
 \rangle}{|\PH\nu^t|},$  $\nu^t$ is
  the Riemannian unit normal along
   $S_t:=\vartheta_t(S)$ and $(\MS)_t$ is the horizontal mean curvature of
   $S_t$. Note that the last identity follows from the
   infinitesimal 1st variation of $\perh$. Using again Stokes' Theorem, it turns out that the 2nd variation of
    $\sigma^{2n+1}\rr\res D$ can be written out as a boundary integral along $S$, i.e.
    \begin{eqnarray}\label{2vavol}II_D(W,\sigma^{2n+1}\rr)=\int_S\left(\widetilde{W}
  (\cn_t)\big|_{t=0}
-\MS\cn^2\right)\,\per.\end{eqnarray}
\begin{corollario}\label{bho}
Let $D\subset\mathbb{H}^n$ be a $C^2$-smooth compact domain. Suppose
$S=\partial D$ has constant horizontal mean curvature and let $\frac{1}{|\PH\nu|}\in
L^2(S,\perh)$. Let $\vartheta_t$ be a volume preserving normal
variation of $S$ having variation vector $W=\varphi\,|\PH\nu|\,
 \nu$ for some smooth function
 $\varphi:S\longrightarrow\R$. Then\begin{equation}\label{volpres2v}
II_S(W,\per)=\int_{S}\left( |\qq \varphi|^2+ \varphi^2\left(
-\|S\cc\|^2\ngr+ 2\frac{\partial\varpi}{\partial\nn^\perp}
-\frac{n+1}{2}\varpi^2 \right)\right)\perh.\end{equation}
\end{corollario}
\begin{proof}For volume preserving
variations, using the 2nd variation formula of volume yields
\[\int_S\left(\MS\widetilde{W}
  (\varphi_t)\big|_{t=0}
-\MS^2\varphi^2\right)\,\per=0\]and the thesis follows by substituting the last identity into
\eqref{Finale2v}.
\end{proof}

We also have the following \textquotedblleft alternative\textquotedblright  version.

\begin{corollario}\label{bho2} Let $n>1$.
Let $D\subset\mathbb{H}^n$ be a $C^2$-smooth compact domain with boundary
$S=\partial D$ of constant horizontal mean curvature. Let $\vartheta_t$ be a volume preserving normal
variation of $S$, having variation vector $W=\varphi\,|\PH\nu|\,
 \nu$, for some  smooth function
 $\varphi:S\longrightarrow\R$ on $S\setminus C_S$. Furthermore, assume that
the differential forms $\chi_1, \chi_2, \chi_3$ are continuous on all of $S$; see Notation \ref{notchi}.
Then \begin{eqnarray*}II_S(W,\per)
=\int_{S}\left(
|\qq\varphi|^2 +\varphi^2\left( -\|S\cc\|^2\ngr+
2\frac{\partial\varpi}{\partial\nn^\perp} -\frac{n+1}{2}\varpi^2
\right)\right)\perh.\end{eqnarray*}
\end{corollario}

The {\it Isoperimetric Functional} in our context can naturally be defined by
\begin{equation}\label{IPP}J (D):=\frac{\per(\partial D)}
{\left(\sigma^{2n+1}\rr(D)\right)^{1-\frac{1}{Q}}},\end{equation}where
$D$ varies over bounded domains in $\mathbb{H}^n$ having
$\mathbf{C}^2$-smooth boundaries. So let $\vartheta_t$ be a
variation of $D$ with variation vector $W$. Differentiating
\eqref{IPP} along the flow $\vartheta_t$, using \eqref{1varvol}
and  \eqref{Finale1v}, yields
\begin{equation}\label{ikik}\frac{d}{d t}J
(\vartheta_t(D))\bigg|_{t=0}=-\frac{1}{\left({\sigma^{2n+1}\rr(D)}\right)^
{1-\frac{1}{Q}}} \int_{\partial D}\MS \cn\,\perh -
\frac{Q-1}{Q}\frac{\per(D)}{\left({\sigma^{2n+1}\rr(D)}\right)^{2-\frac{1}{Q}}}
\int_{\partial D}\cn\,\perh.\end{equation}  By choosing a {\it
volume-preserving variation}. This means that the flow
$\vartheta_t$ associated with $W$ does not change the
 volume, i.e. ${\sigma^{2n+1}\rr(\vartheta_t(D))}={\sigma^{2n+1}\rr(D)}$ for every
$t\in]-\epsilon,\epsilon[$. It follows that the last integral vanishes and, by means of
the Fundamental Lemma of Calculus of Variations, we obtain
the following:
\begin{corollario}\label{mscost}Let $D\subset\mathbb{H}^n$ be a relatively compact domain with
$C^2$-smooth boundary and assume that $D$ is a critical point of
the functional $J(D)$ under volume-preserving variations. Then
$\MS$ must be constant on $\partial D\setminus C_{\partial D}$.
\end{corollario}

\begin{oss}Let $D$ be a critical point of $J(\cdot)$ under volume-preserving variations.  Using
\eqref{ikik} yields
$$\sigma^{2n+1}\rr(D)\,\MS=- \frac{Q-1}{Q}\per(S)$$which implies that
there are no closed compact  $\HH$-minimal hypersurfaces in $\mathbb H^n$; see
\cite{HP2}, \cite{RR}.\end{oss}

Using volume-preserving normal variations, we also get that
\[\frac{d^2}{d^2t}J (\vartheta_t(D))\Big|_{t=0}=\frac{II_{\partial D}(W,
\perh)}{\left({\sigma^{2n+1}\rr(D)}\right)^{1-\frac{1}{Q}}}.\]

The last computation motivates the following:
\begin{Defi}[Stability I]\label{strdef}Let $D\subset\mathbb{H}^n$ be a compact domain with
$C^2$-smooth boundary $S=\partial D$ such that $\frac{1}{|\PH\nu|}\in
L^2(S, \perh)$. We assume that $D$ is a critical
point of $J(D)$  under volume-preserving
variations. We say that $S$ is  a {\rm stable
bounding hypersurface} if $II_{S}(W,\per)> 0$ for every
non-zero volume-preserving normal variation $\vartheta_t$  of
$S$, having variation vector $W=\varphi\,|\PH\nu|\,
 \nu$, where $\varphi:S\longrightarrow\R$ is any smooth function on $S$.
\end{Defi}

Moreover, we propose a weakening of the notion of stability.
\begin{Defi}[Local Stability]\label{strdefloc} Let $D\subset\mathbb{H}^n, n>1,$ be a compact domain with
$C^2$-smooth boundary $S=\partial D$ and assume that $D$ is a critical
point of $J(D)$  under volume-preserving
variations. We say that $S$ is  a {\rm locally stable
bounding hypersurface} if for each $p\in S$ there exists a neighborhood $\Om\subseteq S$ of $p$ such that
 $II_{\Om}(W,\per)> 0$ for every
non-zero volume-preserving normal variation $\vartheta_t$  of
$\Om$ having variation vector $W=\varphi\,|\PH\nu|\,
 \nu$ such that $\varphi: \Om\longrightarrow\R$ is any smooth compactly supported function on $\Om$.
\end{Defi}

\begin{oss}[Radial variations]\label{radstrdef}   Let $D\subset\mathbb{H}^n$ be a compact domain with radial  symmetry with respect to the vertical direction $T$. In this case, a \textquotedblleft natural\textquotedblright class of normal variations can be defined by using radially symmetric functions on $S=\partial D$. More precisely, a useful \textquotedblleft stability test\textquotedblright\,    for the domain $D$
 is that of being stable in the sense of Definition \ref{strdef} for all smooth radial function $\varphi:S\longrightarrow\R$. In this case, we shall say that $D$ is {\rm radially stable}. Clearly, radial stability is just a necessary condition for stability.
\end{oss}

 By applying Corollary \ref{bho2}, the notion of stability can be further generalized.
\begin{Defi}[Stability II]\label{strdef2} Let $D\subset\mathbb{H}^n, n>1,$ be a compact domain with
$C^2$-smooth boundary $S=\partial D$ and assume that $D$ is a critical
point of $J(D)$  under volume-preserving
variations. We say that $S$ is  a {\rm stable
bounding hypersurface} if $II_{S}(W,\per)> 0$ for every
non-zero volume-preserving normal variation $\vartheta_t$  of
$S$ having variation vector $W=\varphi\,|\PH\nu|\,
 \nu$, where $\varphi: S\longrightarrow\R$ is any smooth function on $S\setminus C_S$ such that the differential forms  $\chi_1, \chi_2, \chi_3$ turn out to be continuous on all of $S$; see Notation \ref{notchi}.
\end{Defi}

\section{Isoperimetric Profiles and Stability}\label{STABs} \rm

We already know that
$\Sph$ is (the boundary of) a critical point under
volume preserving variations of the isoperimetric functional $J(\cdot)$. Furthermore, it is not difficult to see that
$\frac{1}{|\PH\nu|}\in L^2(\Sph, {\perh})$; see also Remark \ref{cazfri}.  We start with the following:
\begin{teo}\label{stabilityt} Let $n=1$. The Heisenberg Isoperimetric profile $\Spu$ is a  stable bounding hypersurface in the sense of both Definition \ref{strdef} and  Definition \ref{strdef2}.\end{teo}

\begin{proof} Let  $\vartheta_t$ be any non-zero volume-preserving normal variation having variation vector $W$.
Thus, using Lemma \ref{SeB} together with \eqref{derombar} and the fact that
$(\varpi^\pm)^2 =4\frac{{1-\rho^2}}{\rho^2}$,
yields\[-\|S\cc\|^2\ngr+ 2\frac{\partial\varpi}{\partial\nn^\perp}
-\varpi^2 =0.\]By applying formula \eqref{volpres2v} we obtain
\begin{equation*}
II_{\Spu}(W,\per)=\int_{{\Spu}}|\qq\varphi|^2\perh\geq
0.\end{equation*}If $\int_{{\Spu}}|\qq\varphi|^2\perh=0$, then
$|\qq\varphi|=|\nn^\perp\varphi|=0$. In particular, it follows that $\varphi$ is
constant along any leaf of the so-called  \textit{characteristic foliation} of $\Spu$.
Note also that each leaf joins together North and South poles of
$\Spu$. Hence $\varphi$ is constant along $\Spu$ and since
$\int_{\Spu}\varphi\,\perh=0$, we finally get that $\varphi=0$. Therefore
\[II_{\Spu}(W,\per)>0\]for every non-zero normal variation, as wished; see also Proposition 1.22 in \cite{Monted}.
\end{proof}

\it From now on, we shall study the case  $n>1$. \rm In the general case, we are not able to give a complete proof of the statement valid for $n=1$. Nevertheless, below we shall obtain some partial stability results.\\

\rm  We choose a (non-zero) volume-preserving normal
variation $\vartheta_t$ with variation vector $W$. This means that $\WW$ is parallel to $\nu^t$ for every
$t\in]-\epsilon, +\epsilon[$. As in
Remark \ref{rvvf}, we also assume that $W=\varphi |\PH\nu| \nu$ for some (at least piecewise)
smooth function $\varphi:S\longrightarrow\R$. As already said, this choice implies that $\cn=\frac{\langle W,
\nu\rangle}{|\PH\nu|}=\varphi$. From \eqref{derombar}, Lemma \ref{SeB} and the identity $(\varpi^\pm)^2
=4\frac{{1-\rho^2}}{\rho^2}$, it follows that\[-\|S\cc\|^2\ngr+
2\frac{\partial\varpi}{\partial\nn^\perp} -\frac{n+1}{2}\varpi^2
=-\frac{Q-4}{\rho^2}.\]Hence, using formula \eqref{volpres2v}
yields
\begin{equation}\label{hurin}
II_{\Sph}(W,\per)=\int_{\Sph}\left(|\qq\varphi|^2-\frac{Q-4}{\rho^2}\varphi^2\right)\,\perh\end{equation}
for every smooth function $\varphi:\Sph\longrightarrow \R$ such that
$\int_{\Sph}\varphi\,\perh=0$. Note
that $Q-4=2n-2$.

\begin{oss} In order to study the positivity of the 2nd variation of $\Sph$,
we  shall
set\begin{equation}\label{funz2v}\mathfrak{F}(\varphi):=\int_{\Sph}\left(|\qq\varphi|^2-
\frac{Q-4}{\rho^2}\varphi^2\right)\,\perh\end{equation} and study
the sign of the functional $\mathfrak{F}(\cdot)$ for functions
belonging to the class $\varPhi(\Sph)$ of \it admissible functions; \rm see Definition \ref{fi}.
\end{oss}

Roughly speaking,
we are considering  functions $\varphi\in\cont\ss^2(\Sph\setminus\{\mathcal{N},
\mathcal{S}\})$ which can be \textquotedblleft integrated by parts\textquotedblright\,on $\Sph$.
In using this class, we are including possibly singular
solutions at the poles ${\mathcal N}, {\mathcal
S}$ of $\Sph$, which are the only characteristic points of $\Sph$. Nevertheless, we may apply the horizontal Green's formulas (iii)-(vi) stated in Corollary \ref{Properties}.
Moreover, it is not difficult to realize that the \textquotedblleft right functional
class\textquotedblright   where studying this problem is given by
$$\varPhi_0(\Sph):=\left\{\varphi\in L^2\left(\Sph, \frac{\perh}{\rho^2}\right) :\,\, \varphi\neq 0,\,\,\, |\qq \varphi|\in L^2(\Sph, {\perh}),\,\,
 \int_{\Sph}\varphi\,\perh=0\right\}.$$
For simplicity, in the sequel we shall restrict our study
to functions belonging to the class $$\varPhi_1(\Sph):=\varPhi(\Sph)\cap\varPhi_0(\Sph).$$

\begin{oss}\label{1eigen}Integration by parts in
\eqref{funz2v}
yields\[\mathfrak{F}(\varphi)=-\int_{\Sph}\varphi\left(\lh\varphi
+ \frac{2n-2}{\rho^2}\varphi\right)\,\perh\] for every
$\varphi\in\varPhi_1(\Sph)$. Hence, it becomes natural to study the associated equation
\[\lh\varphi
+ \frac{2n-2}{\rho^2}\varphi=C \qquad(C\in\R).\]Note that we can also consider a non-zero constant $C\in\R$, because $\int_{\Sph}\varphi\,\perh=0$. We already know a
solution to this equation when $C=0$. Indeed, Lemma \ref{ccla}
says that $\Delta\ss\kappa=-\frac{2n-4}{\rho^2}\kappa$, where we recall that
$\kappa=\kappa^\pm=\pm\frac{\sqrt{1-\rho^2}}{\rho}$ along $\Sph^\pm$. Since $\varpi=2\kappa$, using \eqref{derombar}
yields $$\lh\kappa+\frac{2n-2}{\rho^2}\kappa=0$$ which shows that  $\kappa$ is an
eigenfunction, with eigenvalue $\mu=2n-2$, of the  closed \textquotedblleft singular\textquotedblright
eigenvalue problem:$$\lh\varphi + \frac{\mu}{\rho^2}\varphi=0\qquad \mbox{for $\varphi\in\varPhi_1(\Sph),\,(\mu\in\R_+)$}.$$ \end{oss}\vspace{0.5cm} For the sake of completeness, let us first compute 1st and 2nd
variation of $\mathfrak{F}(\cdot)$. So let
$t\in]-\epsilon, \epsilon[$, let $\varphi_1,
\varphi_2\in\varPhi_1(\Sph)$ and consider the \textquotedblleft perturbed  functional\textquotedblright \,
  $\mathfrak{F}\left(\varphi+
t\varphi_1+\frac{t^2}{2}\varphi_2\right)$. Then, the 1st and 2nd
variation of $\mathfrak{F}(\cdot)$ can easily be obtained by
computing the following derivatives:
\[\mathfrak{F}'(\varphi):=\frac{d}{d t}\mathfrak{F}\left(\varphi+
t\varphi_1+\frac{t^2}{2}\varphi_2\right)\bigg|_{t=0},\qquad
\mathfrak{F}''(\varphi):=\frac{d^2}{d
t^2}\mathfrak{F}\left(\varphi+
t\varphi_1+\frac{t^2}{2}\varphi_2\right)\bigg|_{t=0}.\]We thus
have\begin{eqnarray*}\mathfrak{F}'(\varphi)&=&2\int_{\Sph}\left(\langle\qq\varphi,
\qq\varphi_1\rangle-\frac{2n-2}{\rho^2}\varphi\varphi_1\right)\perh,\\\mathfrak{F}''(\varphi)&=&
2\int_{\Sph}\left(\langle\qq\varphi,
\qq\varphi_2\rangle+\langle\qq\varphi_1,
\qq\varphi_1\rangle-\frac{2n-2}{\rho^2}(\varphi_1^2+\varphi\varphi_2)\right)\perh,
 \end{eqnarray*}and, by integrating by parts, we obtain
 \begin{eqnarray*}\mathfrak{F}'(\varphi)&=&-2\int_{\Sph}\varphi_1\left(\lh\varphi+\frac{2n-2}{\rho^2}\varphi
 \right)\perh,\\\mathfrak{F}''(\varphi)&=&
2\int_{\Sph}\left(-\varphi_2\left(\lh\varphi+\frac{2n-2}{\rho^2}\varphi\right)
+ |\qq\varphi_1|^2-\frac{2n-2}{\rho^2}\varphi_1^2\right)\perh.
 \end{eqnarray*}It follows that any critical point of
 $\mathfrak{F}$ (i.e. any solution $\varphi\in\varPhi_1(\Sph)$ to
 $\mathfrak{F}'(\varphi)=0$) solves the equation:
 \[\lh\varphi+\frac{2n-2}{\rho^2}\varphi=C\]for some constant $C\in\R$. Hence, a critical
 point of $\mathfrak{F}$ is a {\it stable minimum} if, and only if, one has\[\mathfrak{F}''(\varphi)=2
 \int_{\Sph}\left(|\qq\varphi_1|^2-\frac{2n-2}{\rho^2}\varphi_1^2\right)\perh\geq
 0\]for all $\varphi_1\in \varPhi_1(\Sph)$. As a straightforward consequence, {\it positivity
 of $\mathfrak{F}$ is equivalent to  positivity of
 $\mathfrak{F}''$}.\\

\begin{oss}Unlike the Riemannian case, for which we refer the reader to
\cite{DC}, the knowledge of the
minimum eigenvalue of $\lh$ on $\Sph$ is not sufficient to solve this problem.
More precisely,
Rayleigh's Inequality says that
\[\lambda_1\leq
\frac{\int_{\Sph}|\qq\varphi|^2\,\perh}{\int_{\Sph}\varphi^2\,\perh}
\qquad\forall\,\,\varphi\in\varPhi(\Sph),\,\,\int_{\Sph}\varphi\,\perh=0,\]where $\lambda_1$ denotes the first non-trivial
eigenvalue of  $\lh$ on $\Sph$. This
implies that\begin{equation*} II_{\Sph}(W,\per)\geq\int_{\Sph}\varphi
^2\left(\lambda_1-\frac{Q-4}{\rho^2}\right)\,\perh,\end{equation*} with
strict inequality unless $\varphi $ is an eigenfunction associated
to $\lambda_1$. But the last integral it is not necessarily
greater than zero.\end{oss}

From now on, we will study the closed eigenvalue problem (singular
at $\mathcal{N}, \mathcal{S}$):
\begin{equation}\label{1eqfinStab}\lh\varphi +
\frac{\mu}{\rho^2}\varphi=0\qquad \forall\,\,\varphi\in
\varPhi_1(\Sph).\end{equation}Here we have to remark that all solutions to this equation must satisfy the following further
compatibility condition:$$\int_{\Sph}\frac{\varphi}{\rho^2}\,\perh=0.$$To see this, it is
sufficient to integrate \eqref{1eqfinStab} over $\Sph$ and
use $\int_{\Sph}\lh\varphi\,\perh=0$. Furthermore, it is a simple consequence
of the horizontal Green formulas discussed in Section \ref{IBPAA},
that the following hold:
\begin{itemize}\item all eigenvalues are positive real numbers;
\item all eigenfunctions can be chosen to be real-valued; \item
eigenfunctions corresponding to distinct eigenvalues are
orthogonal with respect to the \textquotedblleft weighted\textquotedblright inner product
$(\phi_1,\phi_2)_0:=\int_{\Sph}\frac{\phi_1\,\phi_2}{\rho^2}\perh$;\item
all eigenfunctions can be chosen to be orthogonal (note that
eigenvalues with multiplicity will have several eigenfunctions)
with respect to $(\cdot,\cdot)_0$.\end{itemize}\vspace{0.5cm}

\noindent Set now
\[\mathfrak{G}(\varphi):=\frac{\int_{\Sph}|\qq\varphi|^2\,\perh}
{\int_{\Sph}\frac{\varphi^2}{\rho^2}\,\perh}\qquad \forall \,\,\varphi\in \varPhi_1(\Sph).\]
\begin{lemma}\label{qw0} Let
$\mu_1$ be the first eigenvalue of \eqref{1eqfinStab} and consider
the minimization problem:\[{\bf m}:=\min_{\varphi\in
\varPhi_1(\Sph)}\mathfrak{G}(\varphi).\]Assume that the minimum is
achieved, but probably not unique. Then $\bf{m}$ is the first
eigenvalue of   \eqref{1eqfinStab}  and any  minimizer $f$ of $\mathfrak G(\cdot)$ is a corresponding
eigenfunction.
\end{lemma}

\begin{proof}If $f:\Sph\longrightarrow\R$ is a minimizer in $\varPhi_1(\Sph)$, then
$\mathfrak{G}(f)\leq\mathfrak{G}(\varphi)$ for all $\varphi\in
\varPhi_1(\Sph)$. In this case, the real-valued function
$g(\epsilon)=\mathfrak{G}(f+\epsilon\varphi)$ has a minimum at
$\epsilon=0$ and hence $g'(0)=0$. We have
\begin{equation}g'(0)= 2 \frac{
\left(\int_{\Sph}\frac{f^2}{\rho^2}\,\perh
\right)\left(\int_{\Sph}\langle\qq f, \qq \varphi\rangle\,\perh
\right)-\left(\int_{\Sph}\frac{f\varphi}{\rho^2}\,\perh
\right)\left(\int_{\Sph}|\qq f|^2\,\perh \right)}{\left
(\int_{\Sph}\frac{f^2}{\rho^2}\,\perh\right)^2}=0.\end{equation}
Therefore \[\int_{\Sph}\langle\qq f, \qq
\varphi\rangle\,\perh={\bf
m}\,\int_{\Sph}\frac{f\varphi}{\rho^2}\,\perh\]and since
$-\int_{\Sph}\varphi\lh f\,\perh=\int_{\Sph}\langle\qq f, \qq
\varphi\rangle\,\perh$, it follows
that\[\int_{\Sph}\varphi\left(\lh f+{\bf
m}\frac{f}{\rho^2}\right)\,\perh=0\]for all $\varphi\in \varPhi_1(\Sph)$.
Hence
\[\lh f+{\bf m}\frac{f}{\rho^2}=0,\] i.e. $f$ is eigenfunction  of
\eqref{1eqfinStab} with eigenvalue ${\bf m}$. In order to prove
the last claim, let $\mu_i$ be another eigenvalue of
\eqref{1eqfinStab} with corresponding eigenfunction $f_i$. Then
\[{\bf m}\leq \mathfrak{G}(f_i)=-\frac{\int_{\Sph} f_i\lh f_i\,\perh}
{\int_{\Sph}\frac{\varphi^2}{\rho^2}\,\perh}=\mu_i\]and hence
${\bf m}=\mu_1$.\end{proof}

We already know that the function
$\kappa=\kappa^\pm=\pm\frac{\sqrt{1-\rho^2}}{\rho}$ is an
eigenfunction of \eqref{1eqfinStab} with corresponding
eigenfunction $\mu=2n-2$; see  Remark \ref{1eigen}.
{\it Below we shall show that $\mu=\mu^{rad}_1$}, where $\mu^{rad}_1$ denotes the 1st (radial) eigenvalue of \eqref{1eqfinStab}
in the class of all radial  $\varphi\in\varPhi_1(\Sph)$.\\

Along the lines of \cite{Monted}, where a similar method is adopted to study the  \textquotedblleft closed eigenvalue problem  on $\Sph$\textquotedblright for the equation $$\lh \varphi+\lambda\varphi=0,$$as a first step, we shall study the equation
\eqref{1eqfinStab} for radial functions on $\Sph$. Hence, we have
to solve the following O.D.E.:
\begin{equation}\label{radial1eqfinStab}
\varphi_{\rho\rho}''(1-{\rho}^2)+\frac{\varphi'_{\rho}}{\rho}\left(2n
-(2n+1){\rho^2}\right)+
\frac{\mu}{\rho^2}\varphi=0,\end{equation}where $\varphi$ is now a radial function belonging to
$\varPhi_1(\Sph)$; see Remark \ref{RemFinT}. This can be done, exactly as in
\cite{Monted}, by studying the restrictions $\varphi^\pm$ of
$\varphi$ to the hemispheres $\Sph^\pm$, together with some suitable boundary
conditions. For the sake of simplicity, in doing this, we shall assume  $\varphi\in\cont^2(]0, 1])$.

\begin{oss}Remind  that
\[\perh\res\Sph^\pm=\frac{\rho}{2\sqrt{1-\rho^2}}\,dz\res B_1(0),\]
and using spherical coordinates $(\rho,\xi)\in[0,
1]\times\mathbb{S}^{2n-1}$ on $B_1(0)\subsetneq\R^{2n}$ yields
\[dz=\rho^{2n-1}\,d\rho\wedge
d\sigma_{\mathbb{S}^{2n-1}}(\xi).\]The integral conditions
required for belonging to $\varPhi_1(\Sph)$  can then be
rephrased in terms of one-dimensional integrals over the interval
$[0, 1]$, endowed with a \textquotedblleft weight-function\textquotedblright induced by
$\perh$.
\end{oss}

\begin{lemma}\label{qw}Let $\psi$ be an eigenfunction of
\eqref{1eqfinStab} with corresponding eigenvalue $\mu$ and denote
by $\psi_0$ its spherical mean, i.e.
\[\psi_0:=\int_{\mathbb{S}^{2n-1}}\psi\,d\sigma_{\mathbb{S}^{2n-1}}.\]If $\psi_0\neq 0$, then $\psi_0$ is an eigenfunction
of \eqref{radial1eqfinStab} with corresponding eigenfunction
$\mu$.
\end{lemma}

\begin{proof}Analogous to the proof of Claim 2.23 in
\cite{Monted}. First, note that
$$\int_{\Sph}\psi_0\,\perh=\int_{\Sph}\int_{\mathbb{S}^{2n-1}}\psi\,\left(d\sigma_{\mathbb{S}^{2n-1}}\wedge\perh\right)=
\int_{\mathbb{S}^{2n-1}}\underbrace{\Big(\int_{\Sph}\psi
\,\perh\Big)}_{=0}\,d\sigma_{\mathbb{S}^{2n-1}}=0.$$By making use
of \eqref{ffs} (see Remark \ref{RemFinT}) we have
\begin{eqnarray*}\lh\psi&=& (1-\rho^2)\psi''_{\rho\rho}+\frac{2n
-(2n+1)\rho^2}{\rho}\psi'_\rho-
 2\rho\sqrt{1-\rho^2}\psi''_{\zeta \rho} \\&&+\frac{1}{\rho^2}\Delta_{\mathbb{S}^{2n-1}}\psi
-(1-\rho^2)\psi''_{\zeta\zeta}-(Q-1)
\sqrt{1-\rho^2}\psi'_\zeta.\end{eqnarray*}Integrating this
expression along $\mathbb{S}^{2n-1}$, using \eqref{rftr} and the
Divergence Theorem for the Sphere ${\mathbb{S}^{2n-1}}$, yields
\begin{equation}\label{hgty}\int_{\mathbb{S}^{2n-1}}\lh \psi\,d\sigma_{\mathbb{S}^{2n-1}}=
\int_{\mathbb{S}^{2n-1}}\left((1-\rho^2)\psi''_{\rho\rho}+\frac{2n
-(2n+1)\rho^2}{\rho}\psi'_\rho\right)\,d\sigma_{\mathbb{S}^{2n-1}}.\end{equation}
Furthermore, one has
\begin{eqnarray*}\lh\psi_0&=&(1-{\rho}^2)\frac{\partial^2\psi_0}{\partial\rho^2}+\frac{2n
-(2n+1)\rho^2}{\rho}\frac{\partial\psi_0}{\partial\rho}\\&=&\int_{\mathbb{S}^{2n-1}}\left((1-\rho^2)\psi''_{\rho\rho}+\frac{2n
-(2n+1)\rho^2}{\rho}\psi'_\rho\right)\,d\sigma_{\mathbb{S}^{2n-1}}
.\end{eqnarray*}We therefore get that
$$\lh\psi_0=\int_{\mathbb{S}^{2n-1}}\lh \psi
\,d\sigma_{\mathbb{S}^{2n-1}}=-\mu\int_{\mathbb{S}^{2n-1}}\frac{\psi}{\rho^2}\,d\sigma_{\mathbb{S}^{2n-1}}
=-\mu\frac{\psi_0}{\rho^2}$$which proves the thesis.
\end{proof}

Unfortunately, the spherical mean $\psi_0$ of any eigenfunction $\psi$ can be equal to $0$ and so, in general,  we cannot conclude that there are no other eigenvalues, apart from the radial eigenvalues. We need something different.\\

In the next Lemma \ref{qw2} we will use Frobenious' Method; see \cite{GWo}.
Note that if we attempt
 to find the solution of equation \eqref{radial1eqfinStab} in the form of a power series, we have to
 employ the Laurent expansion around $t=0$. So let $m\in \mathbb{Z}$ and assume that
\begin{eqnarray}\label{asasasasasa}\varphi(\rho)=
\rho^{-m}\sum_{l=0}^{+\infty}a_l\rho^l=\sum_{l=0}^{+\infty}a_l\rho^{l-m}.\end{eqnarray}
\begin{lemma}\label{qw2}There exist solutions to \eqref{radial1eqfinStab} of the form
\eqref{asasasasasa} if, and only
if, one has\[\mu_m=m(2n-(m+1))\quad\mbox{for
 every}\,\,\, m=1,...,2n -1.\]These numbers can be
 eigenvalues of \eqref{radial1eqfinStab}, associated with radial eigenfunctions
  belonging to the class $\varPhi_1(\Sph)$,  only if  $m=1,...,n-1$.
 In particular, the first eigenvalue of \eqref{radial1eqfinStab} turns out to be  $\mu_1=Q-4$ and its associated
eigenfunction is given, up to constants, by $\varphi_1=\kappa$; see Remark \ref{1eigen}.
\end{lemma}

\begin{proof}
Since $\varphi'(\rho)=\sum_{l=0}^{+\infty}(l-m)a_l\rho^{{l-m}-1}$
  and $\varphi''(\rho)=\sum_{l=0}^{+\infty}(l-m)({l-m}-1)a_l\rho^{{l-m}-2}$,
  substituting these expressions into \eqref{radial1eqfinStab}
  yields
\begin{eqnarray*}\sum_{l=0}^{+\infty}\left((l-m)({l-m}-1)a_l
(1-\rho^2)+(l-m)a_l\big((2n) -(2n+1){\rho^2}\big)+\mu
a_l\right)\rho^{l-m-2}=0.\end{eqnarray*}So we get that \begin{eqnarray*}\sum_{l=0}^{+\infty}\underbrace{a_l\left((l-m)({l-m}-1+2n)+\mu\right)}_{=:\alpha_l}\rho^{{l-m}-2}=
\sum_{l=0}^{+\infty}\underbrace{a_l\left((l-m)({l-m}+ 2n
\right)}_{=:\beta_l}\rho^{{l-m}}.\end{eqnarray*}From this identity
we infer a system of necessary conditions on the coefficients of
the Laurent expansion of $\varphi$. More precisely, we must
have\begin{eqnarray}\label{arra12}\alpha_0=\alpha_1=0 \quad\mbox{and}\quad
\alpha_{l+2}=\beta_l\quad\mbox{for all}\,\,\, l\in
\mathbb{N}.\end{eqnarray}Since
$\alpha_0=a_0\left(m(m+1-2n)+\mu\right)$ and
$\alpha_1=a_1(-(m-1)(2n-m)+\mu)$ we obtain either
\begin{eqnarray}\label{arra13}a_0=0
 \quad\mbox{or}\quad \mu= m(2n-1-m),\end{eqnarray} or
\begin{eqnarray}\label{arra13}a_1=0 \quad\mbox{or}\quad
\mu=(m-1)(2n-m);\end{eqnarray}furthermore
\[a_{l+2}=a_l\frac{(l-m)(l+2n-m)}{(l+2-m)(l+2n+1-m)+\mu}\qquad\mbox{for all}\,\,\, l\in
\mathbb{N}.\]Note that this procedure  makes possible to write down the solutions
in terms of recurrence relations. From \eqref{arra12} we get that,
if $a_0=a_1=0$, then all coefficients must be zero. Since $\mu>0$, assuming $a_0\neq 0$
(and therefore, $a_1=0$), yields $m\in[2, 2n-1]$, while assuming
$a_1\neq 0$ (and therefore, $a_0=0$) yields $m\in[1, 2n-2]$. This
proves the first claim.\\
\indent For what concerns the second claim, note that any
(radial) solution $\varphi$ belongs to $\varPhi_1(\Sph)$
 only if
$$\varphi\in L^2\left([0, 1], \frac{\rho^{2n-2}}{2\sqrt{1-\rho^2}}\,d\rho
\right),\qquad \varphi'\in L^2\left([0, 1],
\frac{\rho^{2n}}{2\sqrt{1-\rho^2}}\,d\rho \right).$$ Therefore, by an
elementary computation, we get that it must be
\[2n-2(m+1)>-1\Longleftrightarrow m\leq n-1.\]Finally, the last
claim was already known; see Remark \ref{1eigen}.
\end{proof}

Let us state a first consequence. By making use of Lemma \ref{qw0} and Lemma \ref{qw2},
we get that
\[{\int_{\Sph}|\qq\varphi|^2\,\perh}\geq
(2n-2){\int_{\Sph}\frac{\varphi^2}{\rho^2}\,\perh}\qquad\mbox{for
all radial}\,\,\varphi\in \varPhi_1(\Sph).\]Therefore
\[II_{\Sph}(W,\per)=\int_{\Sph}\left(|\qq\varphi|^2-\frac{Q-4}{\rho^2}\varphi^2\right)\,\perh\geq 0\]
for all radial function $\varphi\in \varPhi_1(\Sph)$ and, more precisely,  $II_{\Sph}(W,\per)>0$, unless $\varphi$ is an
eigenfunction of $\mu_1=Q-4$. The radial eigenfunction of $\mu_1$, up to
constants, is  the function
$\kappa^\pm=\pm\frac{\sqrt{1-\rho^2}}{\rho}$ which does not satisfy
Definition \ref{strdef}, because it is singular at the poles. This
proves the following:
\begin{Prop}[Radial stability]The Isoperimetric Profile $\Sph$ turns out to be \rm radially stable \it in the sense of Remark \ref{radstrdef}. More precisely, let $\vartheta_t$ be any normal variation of\, $\Sph$, having variation vector
$W=\varphi |\PH\nu| \nu$, where $\varphi\in\varPhi_1(\Sph)$ is radial. Then, we have $II_{\Sph}(W,\per)>0$ and  $II_{\Sph}(W,\per)=0$ if, and only if, $\varphi$ is an eigenfunction for the first eigenvalue  $\mu_1^{rad}:=\mu_1=Q-4$ of equation \eqref{radial1eqfinStab}.
\end{Prop}

We now discuss some other features of the  general case. We start from a lemma
which is well-known in the classical setting; see \cite{FCS}.

\begin{lemma}\label{zxc}Let $S\subset\mathbb H^n$ be a hypersurface of class $\cont^2$ and let $\Om\subset S$ be any bounded domain. If there exists a function $\psi>0$ on $\Om$ satisfying the equation $\lh \psi= q \psi$, then $$\int_{\Om}\left(|\qq \varphi|^2+ q \varphi^2\right)\,\perh\geq 0$$for all smooth function $\varphi$ compactly supported on $\Om$.
\end{lemma}
\begin{proof}If $\psi>0$ satisfies $\lh \psi= q \psi$ on $\Om$, let us define a new function $\phi:=\log \psi$.
By an elementary calculation we see that $\lh\phi = q- |\qq\phi|^2$. Indeed, one has
\begin{eqnarray*}\lh\phi&=&\div\ss\left(\qq\phi\right)-\varpi\langle\nn^\perp, \qq\phi\rangle\\
&=&\div\ss\left(\dfrac{\qq\psi}{\psi}\right)-\varpi\left\langle\nn^\perp,\dfrac{\qq\psi}{\psi}\right\rangle
\\&=&\left(\dfrac{\Delta\ss\psi}{\psi}-\varpi\left\langle\nn^\perp,\dfrac{\qq\psi}{\psi}\right\rangle\right)-
\dfrac{|\qq\psi|^2}{\psi^2}\\&=& \dfrac{\lh\psi}{\psi}-{|\qq\phi|^2}
\\&=& q-{|\qq\phi|^2}.
\end{eqnarray*}
Now let $\varphi$ be any smooth function with compact support on $\Om$. Multiplying by $-\varphi^2$ both sides of this equation and integrating by parts, yields
\begin{eqnarray}\label{jkhgd}-\int_{\Om}\varphi^2\left(q-|\qq\phi|^2\right)\,\perh=-\int_{\Om}\varphi^2\lh\phi\,\perh=
\int_{\Om}2\varphi\langle\qq\varphi,\qq\phi\rangle\,\perh.
\end{eqnarray}
Now since $$2|\varphi\langle\qq\varphi,\qq\phi\rangle|\leq
2|\varphi||\qq\varphi||\qq\phi|\leq
|\varphi|^2|\qq\phi|^2+|\qq\varphi|^2,$$the thesis follows by
inserting this inequality into \eqref{jkhgd} and then by
cancelling the terms $\int_{\Om}\varphi^2|\qq\phi|^2\,\perh$.
\end{proof}

\begin{corollario}\label{zxc2}Let $\Om\Subset\Sph^+$ or $\Om\Subset\Sph^-$. Then, the following inequality holds
$$\int_{\Om}\left(|\qq \varphi|^2 - \dfrac{2n-2}{\rho^2}\varphi^2\right)\,\perh\geq 0$$for all smooth function $\varphi$ compactly supported on $\Om$.
\end{corollario}
\begin{proof} Setting $q=- \dfrac{2n-2}{\rho^2}$, the thesis follows by applying Lemma \ref{zxc} with $\psi:=\sqrt{\frac{  {1-\rho^2}}{\rho}}$
 \end{proof}

Another easy consequence of Lemma \ref{zxc} is contained in the next:
\begin{corollario}\label{zxc3}Set $\mathbb S^\ast:=\left\{p=\exp(z,t)\in\Sph: \rho=|z|\geq \sqrt{\frac{2n-3}{2n-2}}\right\}$. Then, for every
 $\Om\Subset\mathbb S^\ast$  the following inequality holds
$$\int_{\Om}\left(|\qq \varphi|^2 - \dfrac{2(2n-3)}{\rho^2}\varphi^2\right)\,\perh\geq 0$$for all smooth function $\varphi$ compactly supported on $\Om$.
\end{corollario}
\begin{proof} Choose $q=-\frac{2(2n-3)}{\rho^2}$. Note that the function $\psi:=\frac{2n-2}{2n-3}-\frac{1}{\rho^2}$ is strictly positive on every $\Om\Subset\mathbb S^\ast$. Furthermore, it turns out that $\lh \psi= q \psi$. Then, the thesis follows by applying Lemma \ref{zxc}.
\end{proof}

\begin{oss} In the inequalities of both Corollary \ref{zxc2} and Corollary \ref{zxc3}, the function $\varphi$ is not necessarily zero-mean. In other words, we do not require the validity of the condition $\int_\Om \varphi\,\perh=0$. In a sense, these inequalities are stronger than what one might expect.

\end{oss}

Putting together Corollary \ref{zxc2} and Corollary \ref{zxc3}, we immediately get the following:

\begin{teo}[Local Stability]\label{stabilityt2}  The Isoperimetric profile $\Sph$ is a locally stable bounding hypersurface in the sense of  Definition \ref{strdefloc}.\end{teo}

We end this section with the following:

\begin{oss}[Question]\label{zxc3} If $\mu$ denotes the 1st eigenvalue of \eqref{1eqfinStab}, then is it true that $\mu=2n-2$?
Roughly speaking, is the 1st eigenvalue of \eqref{1eqfinStab}  equal to the first eigenvalue of the radial case?
\end{oss}
Note that a negative answer to this question would automatically imply that Isoperimetric Profiles are unstable.

\subsection{Appendix A: the case of $T$-graphs} \label{tg} \rm
Below we overview  the variational
formulas for the $\HH$-perimeter $\perh$, in  the case of smooth $T$-graphs.\\

Let
$\Om\subseteq\R^{2n}$ be an open set, let $u\in\cont^2(\Om)$
and set
$S:=\{\exp(z, t)\in\mathbb{H}^n: t=u(z)\quad\forall\,\,z\in\Om\}$, i.e. $S$  is the
{\it $T$-graph associated with $u$}. Then,  ${\nu}=\frac{\left(-\nabla_{\R^{2n}}
 u  +\frac{z^{\perp}}{2}, 1\right)}{\sqrt{1+\big\|\nabla_{\R^{2n}} u -\frac{z^\perp}{2}\big\|^2}}$  is the unit normal along $S$ and we have  $\nn=\frac{-\nabla_{\R^{2n}} u +
\frac{z^{\perp}}{2}}{\big\|\nabla_{\R^{2n}} u
-\frac{z^\perp}{2}\big\|}$ and $\varpi=\frac{ 1}
{\big\|\nabla_{\R^{2n}} u
-\frac{z^\perp}{2}\big\|}$. The
$\HH$-perimeter measure $\perh$ on $S$ turns out to be  given by
\begin{eqnarray*}\perh\res\
S=|\PH\nu|\,\sigma^{2n}\rr\res S= \frac{1}{\varpi}\,dz\res \Om
=\left\|\nabla_{\R^{2n}} u -\frac{z^\perp}{2}\right\|\,dz\res
\Om.\end{eqnarray*}The computation of the 1st and 2nd  variation
of
$\perh\res S=\int_{\Om}\left\|\nabla_{\R^{2n}} u -\frac{z^\perp}{2}\right\|\,dz$ can be done by using variations
$\vartheta_t$ which only act along the $T$-direction. So in
particular we are here assuming that the variation vector is given by
$W=\phi\,T$ for some  $\phi\in \cont^1(S)$. Since
$p=\exp(z, u(z))\in S$  for every $z\in\Om$, with a slight abuse
of notation,  we shall also assume that
$\phi:\Om\longrightarrow\R$.
Hence, using $T$-variations one
has\begin{eqnarray*}I_S(\phi\,T,
\perh)=\frac{d}{ds}\bigg|_{s=0}\int_{\Om}\left\|\nabla_{\R^{2n}}(u +
s\phi)
-\frac{z^\perp}{2}\right\|\,dz , \quad II_S(\phi\,T,
\perh)=
\frac{d^2}{ds^{2}}\bigg|_{s=0} \int_{\Om}\left\|\nabla_{\R^{2n}}(u +
s\phi)
-\frac{z^\perp}{2}\right\|\,dz, \end{eqnarray*} for every $\phi\in\cont^1(\Om)$.
An elementary calculation shows that:
\begin{eqnarray}\label{1vT}I_S(\phi\,T,
\perh)&=&-\int_{\Om}\left\langle \nabla_{\R^{2n}}\phi,\left(
\frac{-\nabla_{\R^{2n}}u+\frac{z^\perp}{2}}{\big\|\nabla_{\R^{2n}}
u-\frac{z^\perp}{2} \big\|}\right)\right\rangle\,
dz,\\\nonumber\\\label{2vT}II_S(\phi\,T,
\perh)&=&\int_{\Om}\frac{\big\|\nabla_{\R^{2n}}\phi\big\|^2\big\|\nabla_{\R^{2n}}u-\frac{z^\perp}{2}\big\|^2
-\left\langle\left(\nabla_{\R^{2n}}u-\frac{z^\perp}{2}\right),
\nabla_{\R^{2n}}\phi\right\rangle^2}{\big\|\nabla_{\R^{2n}}
u-\frac{z^\perp}{2}\big\|^3}\, dz.\end{eqnarray}
\begin{oss}\label{yyhhaa}Let
$\Om\subset\R^{2n}$  be a relatively compact open set having
piecewise $\cont^1$-smooth boundary and let $\phi\in\cont^1(\Om)$. Since the integrand in
\eqref{1vT} is everywhere bounded by $\|\nabla_{\R^{2n}}\phi\|$, the 1st variation formula \eqref{1vT} makes sense for every
$T$-graph $u:\Om\longrightarrow\R$ of class $\cont^2$. Furthermore, by using the standard Divergence Theorem, we get\[I_S(\phi\,T,
\perh)=\int_{\Om}\phi\,\div_{\R^{2n}}\left(
\frac{-\nabla_{\R^{2n}}u+\frac{z^\perp}{2}}{\big\|\nabla_{\R^{2n}}
u-\frac{z^\perp}{2} \big\|}\right)\, dz\]for every
$\phi\in\cont_0^1(\Om)$, where $\MS=-\div_{\R^{2n}}\left(
\frac{-\nabla_{\R^{2n}}u+\frac{z^\perp}{2}}{\big\|\nabla_{\R^{2n}}
u-\frac{z^\perp}{2} \big\|}\right).$ The
integrand in \eqref{2vT} can be estimated   by $2 \frac{\|\nabla_{\R^{2n}} \phi\|^2}{\|\nabla_{\R^{2n}}
u-\frac{z^\perp}{2} \|}$. Thus, by assuming
\begin{equation}\label{condvar2}\frac{1}{\big\|\nabla_{\R^{2n}}
u-\frac{z^\perp}{2}\big\|}\in L^1(\Om, dz),\end{equation} the 2nd
variation formula  \eqref{2vT} makes sense for every $T$-graph
$u:\Om\longrightarrow\R$ of class $\cont^2$. It is not
difficult to show that \eqref{condvar2} is equivalent to  $\frac{1}{|\PH\nu|}\in L^1(S, {\sigma^{2n}\rr})$; compare with
Lemma \ref{sv}. Finally, under the same assumptions, one has $II_S(\phi\,T,
\perh)\geq 0$ for every $\phi\in\cont^1(\Om)$.
\end{oss}

\begin{oss}\label{cazfri}For radial $T$-graphs  of class at least $\cont^1$,  condition
\eqref{condvar2} is satisfied. Indeed, one has
${\big\|\nabla_{\R^{2n}}
u-\frac{z^\perp}{2}\big\|}=\sqrt{(u_\rho')^2 +\frac{\rho^2}{4}}$,
and the claim follows since $\frac{1}{\sqrt{(u_\rho')^2
+\frac{\rho^2}{4}}}\leq \frac{2}{\rho}\in L^1(\Om, dz)$.

\end{oss}

In the
case of Isoperimetric Profiles, we  use  \eqref{1vT} and \eqref{2vT}  to give a heuristic proof of
their stability.
\begin{oss}
Let $S=\Sph^\pm$ and  $u=\pm u_0$; see Section \ref{Sez3}. In this
case $\Om=B_1(0)\subset\R^{2n}$ and we have
\begin{eqnarray*}I_{\Sph^\pm}(\phi\,T,
\perh)=-\int_{B_1(0)}\langle\nabla_{\R^{2n}}\phi, z\rangle\,
dz=\int_{B_1(0)}\left(\phi\,\div_{\R^{2n}}z-\div_{\R^{2n}}(\phi
z)\right)\, dz,\end{eqnarray*}where the second equality follows by applying
the usual Divergence Theorem. Therefore
\begin{eqnarray*}I_{\Sph^\pm}(\phi\,T,
\perh)=2n\,\int_{B_1(0)}\phi\, dz\end{eqnarray*}for every $\phi\in
\cont_0^1(B_1(0))$. This shows  that each hemisphere $\Sph^\pm$  is a critical point of $\perh$ belonging to the
class $\Psi=\{\phi\in \cont_0^1(B_1(0)):\int_{B_1(0)}\phi\,
dz=0\}$. Note that the last integral condition gives a volume constraint on the
functional $\perh\res \Sph^\pm$. In other words, we are using
``volume preserving variations''; see Section \ref{volvarIS}. By
Cauchy-Schwartz we obtain
\begin{eqnarray*}II_{\Sph^\pm}(\phi\,T,
\perh)\geq 2 \int_{B_1(0)}\|\nabla_{\R^{2n}}\phi\|^2
{\left(1-\rho^2\right)^{\frac{3}{2}}}\, dz\geq
0,\end{eqnarray*}or, in other words, the stability of each
hemisphere in the class $\Psi$.

\end{oss}

\subsection{Appendix B: further remarks about stability} \label{ptg}

For future purposes, we discuss two conditions concerning stability: a sufficient condition and a necessary  condition.
Let  $D\subset\mathbb{H}^n$  be a compact domain with boundary $S=\partial D$ satisfying either the hypotheses of Corollary \ref{bho} or those of Corollary \ref{bho2}.
Integration by parts yields
\begin{eqnarray*}II_S(W,\per)&=&\int_{S}\left( |\qq\varphi|^2+\varphi^2\left(
-\|S\cc\|^2\ngr+ 2\frac{\partial\varpi}{\partial\nn^\perp}
-\frac{n+1}{2}\varpi^2 \right)\right)\perh\\&=&-\int_{S}\varphi \left(\lh \varphi-\varphi\left(
-\|S\cc\|^2\ngr+ 2\frac{\partial\varpi}{\partial\nn^\perp}
-\frac{n+1}{2}\varpi^2 \right)\right)\perh.\end{eqnarray*}

Therefore, we can turn our attention to a suitable eigenvalue problem for the operator $\lh$.
More precisely, let us consider the following:

\begin{displaymath}
(P)\,\,\left\{%
\begin{array}{ll}
   \lh \varphi=\lambda \varphi\left(
-\|S\cc\|^2\ngr+ 2\frac{\partial\varpi}{\partial\nn^\perp}
-\frac{n+1}{2}\varpi^2 \right)\quad\mbox{on $S$}
 \\\\\int_{S}\varphi\,\perh=0
\end{array}%
\right.\end{displaymath}whenever $\varphi\in \varPhi(S)$, $\varphi\neq 0$; see Definition \ref{fi}. Thus we see that: \begin{itemize} \item \it A sufficient condition for the stability of $D$ is that the first eigenvalue of this problem is greater than, or equal to,  one.\end{itemize}

\rm For what concerns the necessary condition, we first state an integral identity.

\begin{lemma}\label{mio}Let  $S\subset\mathbb{H}^n$ be a $\cont^2$-smooth compact
hypersurface without boundary. Then
\begin{eqnarray}\label{id2}
 \int_S\left(2\varpi^2\frac{\partial\varpi}{\partial\nn^{\perp}}-\dfrac{2n}{3}\varpi^4\right)\,\perh=0,\end{eqnarray}
whenever $\varpi^3\nn^\perp$ is admissible  (for the horizontal divergence formula); see Definition \ref{fi}.\end{lemma}
\begin{proof} We have
\begin{eqnarray*}\int_S\lg(\varpi^3\nn^\perp)\,\perh=\int_S\left(\div\ss(\varpi^3\nn^\perp) -
\varpi\langle \nn^\perp,\varpi^3\nn^\perp\rangle\right)\,\perh=\int_S\left(\div\ss(\varpi^3\nn^\perp) -
\varpi^4\right)\,\perh=0.\end{eqnarray*}Since
\begin{eqnarray*} \div\ss( \varpi^3\nn^{\perp})&=&
3\varpi^2\frac{\partial\varpi}{\partial\nn^{\perp}}+\varpi^3\div\ss(\nn^{\perp})
=3\varpi^2\frac{\partial\varpi}{\partial\nn^{\perp}}+
\varpi^3\left(\sum_{i\in
I\ss}\langle\gs_{\tau_i}\nn^{\perp},\tau_i\rangle\right)\\&=&
3\varpi^2\frac{\partial\varpi}{\partial\nn^{\perp}}-\varpi^3\,
\mathrm{Tr}\big(B\cc(\,\cdot\,, C^{2n+1}\ss\,
 \cdot)\big),\end{eqnarray*}where $C^{2n+1}\ss=C^{2n+1}\cc|_{\HS}$, we get that  \begin{eqnarray}\label{id1}
 \int_S\left(3\varpi^2\frac{\partial\varpi}{\partial\nn^{\perp}}-\varpi^4\right)\,\perh=
 \int_S\left(\mathrm{Tr}\big(B\cc(\,\cdot\,, C^{2n+1}\ss\,
 \cdot)\big)\right)  \varpi \perh.\end{eqnarray}
By Lemma \ref{fios} we have
 $\mathrm{Tr}\big(B\cc(\,\cdot\,, C^{2n+1}\ss\,
 \cdot)\big)=(n-1)\varpi$ and  \eqref{id2} easily follows  from
 \eqref{id1}.\end{proof}

Now we apply Lemma \ref{mio} together with a special choice of the variation vector $W$. Here we have to assume the validity of Corollary \ref{bho2} for a variation $\vartheta_t$, having variation vector $W=\varpi |\PH \nu| \nu$.  Note that $\varpi$ is a 0-mean function on $S$ with respect to the measure $\perh$ and that $\varpi$ is smooth out of $C_S$.
Moreover, let us suppose that the vector field  $\varpi^3\nn^\perp$ is admissible  (for the horizontal divergence formula); see Definition \ref{fi}.
It follows that
\begin{eqnarray*}II_S(W,\per)&=&\int_{S}\left( |\qq\varpi|^2+\varpi^2\left(
-\|S\cc\|^2\ngr+ 2\frac{\partial\varpi}{\partial\nn^\perp}
-\frac{n+1}{2}\varpi^2 \right)\right)\perh\\&=&\int_{S}\left( |\qq\varpi|^2-\varpi^2\left(
 \|S\cc\|^2\ngr+ \left(\dfrac{3-n}{6}\right)\varpi^2 \right)\right)\perh. \end{eqnarray*}

\begin{itemize} \item \it
Under our current assumptions, a necessary condition for the stability of $D$ is given by the following geometric inequality:
\begin{equation}\int_{S} |\qq\varpi|^2\,\perh \geq \int_S\varpi^2 \left(\|S\cc\|^2\ngr+ \left(\dfrac{3-n}{6}\right)\varpi^2 \right) \perh. \end{equation}
 \end{itemize}We stress that, in the case of the Isoperimetric Profile $\Sph$, this inequality  is in fact an identity.

\rm

{\footnotesize \noindent Francescopaolo Montefalcone:
\\Dipartimento di Matematica Pura e Applicata\\Universit\`a degli Studi di Padova,\\
  Via Trieste, 63, 35121 Padova (Italy)\,
 \\ {\it E-mail address}:  {\textsf montefal@math.unipd.it}}

\end{document}